\documentclass[leqno,10pt]{article} 
\usepackage{amsmath, amssymb}
\usepackage{amsthm} 
\usepackage{amssymb}
\usepackage{mathrsfs}
\usepackage{esint}
\usepackage{amssymb, url, color, graphicx, amscd, mathrsfs}
\usepackage[colorlinks=true, bookmarks=true, pdfstartview=FitH, pagebackref=true, linktocpage=true, linkcolor = magenta, citecolor = blue]{hyperref}
\usepackage{graphicx}
\usepackage{times}
\usepackage{mathtools} 

\theoremstyle{plain} 
\newtheorem{theorem}{\indent\sc Theorem}[section]
\newtheorem{lemma}[theorem]{\indent\sc Lemma}
\newtheorem{corollary}[theorem]{\indent\sc Corollary}
\newtheorem{proposition}[theorem]{\indent\sc Proposition}

\theoremstyle{definition} 

\newtheorem{remark}[theorem]{\indent\sc Remark}

\numberwithin{equation}{section}
\DeclareMathOperator{\Ric}{Ric}
\DeclareMathOperator{\tr}{trace}

\DeclareMathOperator{\Ent}{Ent}

\makeatletter
\def\address#1#2{\begingroup
\noindent\parbox[t]{7.8cm}{%
\small{\scshape\ignorespaces#1}\par\vskip1ex
\noindent\small{\itshape E-mail address}%
\/: #2\par\vskip4ex}\hfill%
\endgroup}%
\makeatother
\title{{Spectral $\Gamma_2$ inequalities for the $V$-Laplacian with evolving weighted measures}} 
\author{
\textsc{Ha Tuan Dung, Yi Li} 
}
\date{} 
\begin{document}

\maketitle

\footnote{ 
2020 \textit{Mathematics Subject Classification}.
Primary 58J50; Secondary 53C21, 58J35.}
\footnote{ 
\textit{Key words and phrases}. $V$-Laplacian, Tsallis entropy, Shannon entropy, Sobolev inequalities, Bakry-\'Emery tensors.
}
	\begin{abstract}
In this paper, we establish spectral Bakry-\'{E}mery $\Gamma_2$ criteria and their modified counterparts for the $V$-Laplacian on closed Riemannian manifolds equipped with evolving weighted measures. The estimates yield jointly optimized spectral bounds and sharp endpoint constants. When the measures evolve according to the Fokker-Planck equation, these estimates yield monotonicity formulas for the Fisher information associated with Shannon and Tsallis entropies, together with quantitative propagation of Poincaré, logarithmic Sobolev, and nonlinear interpolation inequalities. In the stationary reversible setting, our results specialize to the Witten Laplacian and recover the unweighted Cai-Wang estimates \cite{Cai}.
	\end{abstract}
\section{Introduction}\label{S1}
\subsection{Background}
\quad 

Logarithmic Sobolev inequalities constitute both fundamental analytical tools and important objects of investigation in geometric analysis, probability theory, and the study of diffusion semigroups \cite{Beck, Bob,  Bren,  Gro1, Gro2, Led, Otto, Roth}.
Logarithmic Sobolev inequalities were introduced by Gross \cite[Section 2]{Gro1}, who established their fundamental connection with hypercontractivity. On a closed, connected Riemannian manifold $(M^n,g)$
of dimension $n\geq 2$, one considers the optimal constant $\alpha_M$ such that 
$$\fint_M |\nabla u|^2\,dv
\geq \frac{\alpha_M}{2}\left[\fint_M u^2\log u^2\,dv-\left(\fint_M u^2\,dv\right)
\log\left(\fint_M u^2\,dv\right)
\right]$$
for every positive $u\in \mathcal{C}^\infty(M)$. 
A classical approach to establishing logarithmic Sobolev inequalities is through the Bakry-\'Emery $\Gamma_2$ calculus \cite{Bak1}. Recall that, for the
Laplace-Beltrami operator $\Delta$, the carr\'e du champ operator $\Gamma$ is
defined by
$$
\Gamma(u, v)=\frac{1}{2}[\Delta(u v)-u \Delta v-v \Delta u]=\langle\nabla u, \nabla v\rangle.
$$
while its iterated form $\Gamma_2$ is given by 
$$
\Gamma_2(u, v)=\frac{1}{2}\left[\Delta \Gamma(u, v)-\Gamma(u, \Delta v)-\Gamma(v, \Delta u)\right] .
$$
In particular, by the Bochner-Weitzenb\"ock formula,
$$
\Gamma_2(u):=\Gamma_2(u, u)=\frac{1}{2} \Delta|\nabla u|^2-\langle\nabla u, \nabla \Delta u\rangle=\left|\nabla^2 u\right|^2+\Ric\left(\nabla u, \nabla u\right),
$$
where $\Ric$ is the Ricci tensor of the manifold $(M^n, g)$.
Under the condition
$\Ric\geq (n-1)g,$ the Bakry-\'Emery criterion yields 
$\alpha_M\geq n.$
This estimate is sharp on the standard sphere. The Ricci curvature lower bound, however, does not capture finer spectral information about the manifold. If $\lambda_1 (\Delta)$ denotes the first positive eigenvalue of
$-\Delta$, Rothaus \cite[Theorem 1]{Roth1} obtained the improved lower bound
$\alpha_M\geq\frac{4n}{(n+1)^2}\lambda_1 (\Delta)+\frac{(n-1)^2}{(n+1)^2}n$
under the same Ricci curvature assumption. 
A related problem is to improve the constant directly at the level of
the Bakry-\'Emery $\Gamma_2$ criterion. Let $\Lambda_M$ denote the
largest constant $C$ such that
$$
\int_M h\left[\left|\nabla^2 \log h\right|^2+\Ric(\nabla \log h, \nabla \log h)\right] dv \geq C \int_M h|\nabla \log h|^2 dv
$$
for every positive smooth function  $h$. The Bakry-\'Emery semigroup
argument implies 
$\alpha_M\geq \Lambda_M.$
Thus, a lower bound for $\Lambda_M$ provides a direct route to a
logarithmic Sobolev inequality through the heat flow.

Motivated in part by the role of the $\Gamma_2$ criterion in the study of
the Landau equation by Guillen and Silvestre \cite[Section 9]{Gui},
Ji \cite[Theorem 2.3]{Ji} obtained a quantitative spectral improvement of the
optimal $\Gamma_2$ constant. In particular, under 
$\Ric\geq (n-1)g,$
Ji proved
$\Lambda_M \geq\frac{4n-1}{n(n+2)}\lambda_1 (\Delta)
+\frac{(n-1)^2}{n(n+2)}\,n.$
This result improves the classical Bakry-\'Emery lower bound at the level of the $\Gamma_2$ criterion itself.  More recently, Cai and Wang \cite[Theorem 2]{Cai}
established a one-parameter family of weighted
Bakry-\'Emery $\Gamma_2$ inequalities which recovers
Ji's estimate after a suitable rescaling of the test
functions and passage to the limit $s\longrightarrow-\infty$. More precisely,
under the assumption 
$\Ric\geq (n-1)g,$
they obtained, for an appropriate range of the parameter $s$, lower bounds
for integrals of the form
$\int_M u^s \Gamma_2 (u)\,dv.$
Cai and Wang also established a modified weighted $\Gamma_2$ inequality
involving 
$\Gamma_2 (u)-\frac{|\nabla u|^2}{u}\Delta u.$
The corresponding modified weighted $\Gamma_2$ integral arises in the
second-derivative formula for the Tsallis entropy \cite{Tsal, Tsal1} along the heat flow.
This leads to a differential inequality relating the first and second
derivatives of the Tsallis entropy and, after integration along the heat
flow, to a family of Sobolev inequalities. In this family, the case $q=1$
corresponds to the Poincar\'e inequality, while the limit $q\longrightarrow 2$ yields
a logarithmic Sobolev inequality. Moreover, the resulting constants are sharp on
the model sphere. 
For broader context, Tsallis entropy functionals are also connected
with certain nonlinear Fokker-Planck and porous medium equations \cite{Kan, Plas}. In
particular power-law models, such equations may admit $q$-Gaussian
self similar or stationary solutions and entropy or free energy structures
related to Tsallis entropy \cite{Schaw, SRT}. Related generalized entropy structures also appear from the gradient flow viewpoint in optimal transport. In particular, Kell \cite[Theorem 4.7]{Kell} identified the $q$-heat flow with the gradient flow of a Renyi entropy functional in the $p$-Wasserstein space.

Motivated by the preceding developments, the main objective
of this paper is to establish spectral Bakry-\'{E}mery
$\Gamma_2$ criteria for the $V$-Laplacian in the setting of evolving weighted measures. Let $(M^n,g)$ be a closed, connected Riemannian manifold of dimension $n\geq 2$.  Given an arbitrary smooth vector field $V\in \mathcal{C}^\infty(TM)$, the associated $V$-Laplacian is defined by 
$$\Delta_V:=\Delta+\langle V,\nabla\cdot\rangle,$$
where $\Delta$ denotes the
Laplace-Beltrami operator. The $V$-Laplacian arises naturally at the intersection of Riemannian geometry, diffusion theory, and the analysis of second-order elliptic operators. Its probabilistic background is closely related to Kolmogorov's theory of Markov diffusions and the associated forward and backward equations \cite{Kol}. From an analytical perspective, the $V$-heat equation
$\frac{\partial}{\partial t} u=\Delta_Vu$
provides a setting in which the interaction between curvature, drift, and the behavior of positive solutions can be investigated through gradient estimates and Harnack inequalities. In the absence of drift, the classical estimates of Cheng and Yau \cite{ChengYau} show how lower Ricci curvature bounds control the logarithmic gradient of positive harmonic functions. For the heat equation, Li and Yau \cite{LiYau} obtained differential estimates coupling the spatial gradient of $\log u$ with its time derivative, leading to parabolic Harnack inequalities. Hamilton \cite{Ham1} established complementary gradient estimates for bounded positive heat solutions, relating $|\nabla\log u|^2$ to $\log \frac{A}{u}$, where $A$ is an upper bound for $u$. Together, these results illustrate how curvature controls the spatial variation and temporal behavior of positive solutions. 
For the $V$-Laplacian, the second author \cite{Li1} developed corresponding Cheng-Yau, Li-Yau, and Hamilton-type estimates under suitable Bakry-\'{E}mery Ricci curvature lower bounds. The same work also established local and global upper bounds for the Hessian of bounded positive solutions of the $V$-heat equation on closed manifolds. These developments are closely related to the broader theory of Markov diffusion semigroups, curvature-dimension conditions, and functional inequalities developed by Bakry, Gentil, and Ledoux \cite{Bak2}. 

The Bakry-\'{E}mery $\Gamma_2$ calculus relates
curvature conditions to functional inequalities
for Markov diffusion semigroups \cite{Bak2}.
For the $V$-Laplacian, the relevant curvature tensors
are defined by
$$
\Ric_V:=\Ric-\frac{1}{2} \mathcal{L}_V g, \quad \Ric_V^N:=\Ric_V-\frac{V^{\flat} \otimes V^{\flat}}{N-n},
$$
for any number $N>n$, where $\mathcal L_Vg$ is the Lie derivative of $g$
along $V$, and $V^\flat$ is the metric dual of $V$.
For $u\in\mathcal C^\infty(M)$, the
Bochner-Weitzenb\"ock formula
\cite[Lemma 2.1]{Li1} gives
\begin{align}
	\Gamma_{2,V}(u):=\frac12\Delta_V|\nabla u|^2
	-\left\langle\nabla u,\nabla\Delta_Vu\right\rangle=|\nabla^2u|^2+\Ric_V(\nabla u,\nabla u).
\end{align}
Together with the Hessian trace estimate and the
Cauchy-Schwarz inequality, this yields the
curvature-dimension inequality
\begin{equation}\label{p1.2}
	\Gamma_{2,V}(u)
	\geq
	\frac1N(\Delta_Vu)^2
	+\Ric_V^N(\nabla u,\nabla u).
\end{equation}
If $V$ is Killing, then $\mathcal L_Vg=0$, so
$\Ric_V=\Ric$ and $\Gamma_{2,V}=\Gamma_2$.
Meanwhile, $\Ric_V^N$ still
depends on $V$ through the quadratic term
$-\frac{V^b \otimes V^b}{N-n}.$ 
A nonzero Killing drift also shows that invariance
of the volume measure does not imply self-adjointness
of $\Delta_V$ in $L^2(M,dv)$.
Taking the trace of $\mathcal L_Vg=0$ gives
$\operatorname{div}V=0$.
Since $M$ is closed, the integration by parts yields
$$
\int_M \Delta_V \varphi d v=-\int_M \varphi \operatorname{div} V d v=0
$$
for all $\varphi\in \mathcal{C}^\infty(M).$
Thus, $dv$ is invariant under the diffusion generated
by $\Delta_V$.
The formal adjoint with respect to $dv$, however, is 
$\Delta_V^*=\Delta-\langle V,\nabla\cdot\rangle.$
For $V\not\equiv0$, this differs from $\Delta_V$,
which therefore is not self-adjoint in $L^2(M,dv)$. 
The deformation tensor $\frac12\mathcal L_Vg$ also underlies the geometric heat flow introduced by the second author and Liu \cite{Li2}.
Their flow is the negative $L^2$-gradient flow of
the deformation energy
$X\mapsto\frac14\int_M|\mathcal L_Xg|^2\,dv$,
whose critical points are precisely the Killing
vector fields.

Extending the spectral $\Gamma_2$ estimates of Ji \cite{Ji} and Cai-Wang \cite{Cai} to the setting of a general drift and a prescribed weighted measure involves two difficulties.
The first is that $\Delta_V$ need not be symmetric
with respect to this measure. 
For a smooth positive density $\rho_t$ and
$d\mu_t=\rho_t\,dv$, symmetry in $L^2(M,d\mu_t)$
at a fixed time $t$ is equivalent to
$V=\nabla\log\rho_t$.
In particular, when $V$ is not a gradient vector field,
no smooth positive reference measure makes $\Delta_V$
symmetric.
Consequently, the classical self-adjoint Rayleigh-Ritz
principle cannot be applied directly to $\Delta_V$
in the general setting considered here.
The issue addressed here is to retain its quantitative spectral
refinement for a general drift and a prescribed evolving measure. 
A second arises from the interaction
between the drift and the curvature structure.
Although the Bakry-\'{E}mery Bochner identity
remains valid for arbitrary smooth vector fields, the curvature tensor  $\Ric_V$ contains the deformation tensor
$\frac12\mathcal L_Vg$, whose sign is not
controlled by a lower bound on the ordinary
Ricci curvature alone. Furthermore, the
standard Hessian trace inequality yields
$$
\left|\nabla^2 u\right|^2 \geq \frac{1}{n}\left(\Delta_V u-\langle V, \nabla u\rangle\right)^2.
$$ Expressing the Hessian trace in terms of $\Delta_Vu$ introduces additional drift-dependent terms. The  inequality \eqref{p1.2} encodes these terms through $\Ric_V^N$. For the nonlinear weighted integral estimates considered here, the reference measure also enters through the integration by parts formula and contributes additional terms involving its logarithmic density.

Our approach uses the symmetric diffusion operator associated with each reference measure, which lets us apply the variational characterization of its spectral gap. This gap, together with the logarithmic gradient of the reference density, provides the analytic input for our estimates. The curvature contribution combines the weighted Ricci tensor of this diffusion with the deformation tensor of the relative drift, namely the difference between $V$ and the drift of the symmetric diffusion. These quantities determine the explicit constants in the nonlinear weighted $\Gamma_2$ inequalities and their modified counterparts. For the evolution results, we consider smooth positive probability densities $\rho_t$ and the associated measures $d \mu_t=\rho_t d v$. Their evolution is governed by the forward Fokker-Planck equation associated with $\Delta_V$ 
\begin{align}\label{q1.2}
\partial_t \rho_t=\Delta_V^* \rho_t=\Delta \rho_t-\operatorname{div}\left(\rho_t V\right),
\end{align}
where $\Delta_V^*$ denotes the formal adjoint of $\Delta_V$
with respect to the Riemannian volume measure $dv$.
On the closed manifold $M$, this equation preserves the normalization $\int_M\rho_t\,dv=1$ whenever it holds initially. Coupling this evolution with positive solutions of the backward $V$-heat equation, we obtain differential inequalities
for Shannon and Tsallis entropies. These inequalities, in turn,
yield quantitative propagation of Poincar\'e, logarithmic
Sobolev, and Sobolev interpolation inequalities along the
measure flow.

\subsection{Main results}
\quad We now state the main results of the paper. Throughout this section,
let $(M^n,g)$ be a closed connected Riemannian manifold of dimension
$n\geq2$, let $V\in \mathcal{C}^\infty(TM)$ be a smooth vector field, and consider the $V$-Laplacian 
$\Delta_V:=\Delta+\langle V,\nabla\cdot\rangle.$
Let $\rho_t>0$ be a smooth family of probability densities on
$M\times[0,T]$, and put
\begin{align}\label{k1.3}
d \mu_t:=\rho_t d v, \quad \int_M d \mu_t=1.
\end{align}
The spectral estimates below are statements at a fixed time and do
not require an evolution equation for $\rho_t$. For the entropy and
interpolation results, we shall assume that $\rho_t$ solves the
Fokker-Planck equation \eqref{q1.2}.

For each $t\in [0, T]$, set
$$W_t:=V-\nabla \log \rho_t, \quad A_t:=\Delta+\left\langle \nabla \log \rho_t, \nabla \cdot\right\rangle, \quad S_t:=\frac{1}{2} \mathcal{L}_{W_t} g .
$$
Then, $\Delta_V=A_t+W_t\cdot\nabla$, and $A_t$ is symmetric with
respect to $d\mu_t,$ that is
$$
\int_M \varphi A_t \psi d \mu_t=-\int_M\langle\nabla \varphi, \nabla \psi\rangle d \mu_t, 
$$
for all $\varphi,\psi\in \mathcal{C}^\infty(M)$.
Moreover,  $-A_t$ is a nonnegative self-adjoint operator
on $L^2(M,d\mu_t)$ and 
$$
\Ric_{A_t}:=\Ric-\nabla^2 \log \rho_t, \quad \Ric_V=\Ric_{A_t}-S_t .
$$
Denote by $\lambda_1(t)$ the first positive eigenvalue of  the operator $-A_t$,
and write $H_t:=\|\nabla\log\rho_t\|_{L^\infty(M)}$. 
We next introduce a function which will appear repeatedly in the
spectral constants below. For
$
0<a<\alpha<\frac{n}{n-1}, \lambda>0,$ and $L \geq 0,
$
set
$
d_\alpha:=n-(n-1) \alpha, \chi_{a, \alpha}:=\frac{1}{n(n-1)\left(a^{-1}-\alpha^{-1}\right)},
$
and define
\begin{equation}\label{eq:1.3}
\widehat{\Psi}_{\alpha, a}(\lambda, L):=\min _{\substack{x \geq \sqrt{\lambda} \\ 0 \leq y \leq L}}\left\{\frac{d_\alpha}{n} x^2-\frac{2 \alpha}{n} x y+\left(\frac{\alpha}{n}-\chi_{a, \alpha}\right) y^2\right\} .
\end{equation}
We also introduce
$$
\mathcal{R}_{\alpha, t}:=\alpha \Ric_{A_t}-S_t=\alpha \Ric_V+(\alpha-1) S_t, \quad \mathfrak{r}_\alpha(t):=\inf _{\substack{x \in M, |\xi|_g=1}}\left(\mathcal{R}_{\alpha, t}\right)_x(\xi, \xi),
$$
and put
$
\widehat{\mathcal{C}}_t(a, \alpha):=\widehat{\Psi}_{\alpha, a}\left(\lambda_1(t), H_t\right)+\mathfrak{r}_\alpha(t) .
$

Our first result gives a spectral lower bound for the $V$-Bakry-\'Emery $\Gamma_2$ integral in terms of the spectral gap of $A_t$, the density gradient, and the deformation tensor of $W_t$.

\begin{theorem}\label{main1}
	Let $(M^n,g)$ be a closed, connected Riemannian manifold
	of dimension $n\geq 2$, equipped with a smooth vector field
	$V\in\mathcal{C}^\infty(TM)$ and the family of probability
	measures $(d\mu_t)_{t\in[0,T]}$ defined by
	\eqref{k1.3}.
	For $s \in(-\infty, 0) \cup\left(\frac{4(n+2)}{4 n-1}, \infty\right),$
	set 
	$\alpha_s:=\frac{n(n-1) s}{(n+2)(n s-4)}.$ 
	Then, for every
$\alpha \in\left(\alpha_s, \frac{n}{n-1}\right)$, every $t\in[0,T]$,
	and every positive function $w\in \mathcal{C}^\infty(M)$, we have
	\begin{align}\label{e1.4}
	\int_M w^s\left[\left|\nabla^2 w\right|^2+\Ric_V(\nabla w, \nabla w)\right] d \mu_t \geq \widehat{\mathcal C}_t(\alpha_s,\alpha) \int_M w^s|\nabla w|^2 d \mu_t.
	\end{align}
\end{theorem}
\begin{remark}\label{r1.2}
	In the stationary reversible case, we can write
$\rho_t \equiv \frac{e^{-f}}{\int_M e^{-f} d v}, V=-\nabla f.$
	Then, we see that  $d \mu_t \equiv d \nu_f:=\frac{e^{-f} dv}{\int_M e^{-f} d v},$ and 
	$A_t=\Delta_f, S_t=0, \Ric_V=\Ric_f,$
	where
	$\Delta_f:=\Delta-\langle\nabla f,\nabla\cdot\rangle, \Ric_f:=\Ric+\nabla^2f.$ 	Suppose that $\Ric_f\geq\kappa g$ and
	$|\nabla f|\leq L$.
	Since $H_t\leq L$ and
	$\mathfrak r_\alpha(t)\geq\alpha\kappa$,
	Theorem \ref{main1} gives
	\begin{align}\label{e1.5}
	\int_M w^s\left[\left|\nabla^2 w\right|^2+\Ric_f(\nabla w, \nabla w)\right] d \nu_f \geq \widehat{C}_{s, f}(\alpha) \int_M w^s|\nabla w|^2 d \nu_f,
	\end{align}
	where
$\widehat{C}_{s, f}(\alpha):=\widehat{\Psi}_{\alpha, \alpha_s}\left(\lambda_{1, f}, L\right)+\alpha \kappa.$
	Here, $\lambda_{1,f}$ is the first positive eigenvalue of
	$-\Delta_f$. 
	If $f$ is constant, we may take $L=0$, and
$
\widehat{\Psi}_{\alpha, \alpha_s}\left(\lambda_1, 0\right)=\left(1-\frac{n-1}{n} \alpha\right) \lambda_1 .
$
For every $s$ in the interior of the admissible Cai-Wang range , letting $\alpha \longrightarrow \alpha_s$ gives
$$
\int_M w^s\left[\left|\nabla^2 w\right|^2+\Ric(\nabla w, \nabla w)\right] d v \geq\left[\left(1-\frac{n-1}{n} \alpha_s\right) \lambda_1+\alpha_s \kappa\right] \int_M w^s|\nabla w|^2 d v .
$$
Thus, when $\kappa=n-1$, the Cai-Wang estimate \cite[Theorem 2]{Cai} is recovered away from the positive endpoint. The case $s=0$ is obtained separately (see Lemma \ref{l2.1}), while the endpoint $s=\frac{4(n+2)}{4 n-1}$ is recovered directly from the unweighted argument, since all drift terms vanish when $f$ is constant.
\end{remark}
\begin{remark}
The joint minimization in \eqref{eq:1.3} yields
\begin{equation}\label{l1.6}
	\widehat{\Psi}_{\alpha,a}(\lambda,L)
	\geq
	\min_{\substack{x\geq\sqrt{\lambda}\\0\leq y\leq L}}
	\left\{
	\frac{d_\alpha}{n}x^2
	-\frac{2\alpha}{n}xy
	+\frac{\alpha}{n}y^2
	\right\}
	-\chi_{a,\alpha}L^2.
\end{equation}
The right-hand side is obtained by bounding the correction
term $-\chi_{a,\alpha}y^2$ below by $-\chi_{a,\alpha}L^2$.
Proposition  \ref{p1.3} gives an explicit formula for the
minimum and characterizes strict inequality in
\eqref{l1.6}.
Optimizing over $\alpha$ gives the best bound within
this family, without establishing sharpness for the
underlying nonlinear inequality.
\end{remark}

Theorem  \ref{main1} also yields a Shannon-type estimate
through the logarithmic rescaling $w_s=h^{\frac{1}{s}}$
as $s\longrightarrow-\infty$. Specifically, for each $t\in[0,T]$
and every positive function $h\in\mathcal C^\infty(M)$,
\begin{equation}\label{p1.9}
	\int_M h\,\Gamma_{2,V}(\log h)\,d\mu_t
	\geq
	\widehat C_{\mathrm{Sh},V}^*(t)
	\int_M h|\nabla\log h|^2\,d\mu_t,
\end{equation} 
where $\widehat C_{\mathrm{Sh},V}^*(t)
:=
\sup_{\alpha\in\left(\frac{n-1}{n+2}, \frac{n}{n-1}\right)}
\widehat{\mathcal C}_t \left(\frac{n-1}{n+2},\alpha\right).$
To interpret this estimate in terms of entropy,
assume additionally that the densities $\rho_t$
satisfy the equation \eqref{q1.2},
and consider the backward Kolmogorov problem
\begin{equation}\label{eq:1.9}
	\begin{cases}
		\partial_tu+\Delta_Vu=0,
		&\text{on }M\times(0,T),\\
		\quad\quad u(\cdot,T)=u_T,
		&\text{on }M,
	\end{cases}
\end{equation}
where $u_T\in\mathcal C^\infty(M)$ is strictly positive
and $\int_Mu_T\,d\mu_T=1$.
Well-posedness, positivity and mass conservation for the
forward equation are established in the Appendix.

Duality preserves $\int_Mu(\cdot,t)\,d\mu_t=1$.
With all derivatives taken in the forward
Fokker-Planck time $t$, the Shannon entropy relative
to the instantaneous reference measure, 
$$\mathcal S_V(u,t):=-\int_Mu\log u\,d\mu_t,$$
satisfies
$$
\frac{d}{d t} \mathcal{S}_V(u, t)=-\int_M u|\nabla \log u|^2 d \mu_t, \quad \frac{d^2}{d t^2} \mathcal{S}_V(u, t)=-2 \int_M u \Gamma_{2, V}(\log u) d \mu_t.
$$
Consequently, \eqref{p1.9} gives
\begin{align}\label{e1.11}
\frac{d^2}{dt^2}\mathcal S_V(u,t)
-2\widehat C_{\mathrm{Sh},V}^*(t)
\frac{d}{dt}\mathcal S_V(u,t)
\leq0.
\end{align}
We next extend this entropy inequality to the
one-parameter family of Tsallis entropies. For $p\in(0,1)\cup(1,\infty)$, define 
$$T_{p,V}(u,t)
:=\frac{\displaystyle\int_Mu^p\,d\mu_t-1}{1-p},$$
where $\int_Mu\,d\mu_t=1$. Since
$\lim_{p\longrightarrow1}T_{p,V}(u,t)=\mathcal S_V(u,t)$, we set
$T_{1,V}:=\mathcal S_V$.
The second variation of $T_{p,V}$, however, is no longer governed
by the usual $\Gamma_2$ expression.  Setting 
$s:=\frac2{p-1}, U:=u^{\frac{p-1}{2}}=u^{\frac{1}{s}},$
a direct computation gives
$$
\frac{d}{d t} T_{p, V}(u, t)=-p s^2 \int_M U^s|\nabla U|^2 d \mu_t,
$$
and
\begin{align}\label{o1.9}
	\frac{d^2}{dt^2}T_{p,V}(u,t)
	=-2ps^2\int_MU^s\bigl[
	&|\nabla^2U|^2+\Ric_V(\nabla U,\nabla U) -U^{-1}|\nabla U|^2A_tU\bigr]d\mu_t.
\end{align}
Thus, the Tsallis deformation introduces the additional term
$-U^{-1}|\nabla U|^2A_tU$. The operator in this term is the
symmetric operator $A_t$, because the integration by parts is
performed with respect to the measure $d\mu_t$.
This calculation motivates the following modified spectral $\Gamma_2$ estimate.
We now return to arbitrary smooth positive probability densities $\rho_t$. Here, the estimate holds at each fixed time without an evolution assumption.

\begin{theorem}\label{main2}
	Let $(M^n,g)$ be a closed, connected Riemannian manifold
of dimension $n\geq 2$, equipped with a smooth vector field
$V\in\mathcal{C}^\infty(TM)$ and the family of probability
measures $(d\mu_t)_{t\in[0,T]}$ defined by
\eqref{k1.3}. For $s \in\left(-\infty,-\frac{2\left(2 n^2+1\right)}{4 n-1}\right) \cup(2, \infty),$ set $\beta_s:=\frac{(n-1)(s-2)}{(n+2)(s+2)}.$ 
	Then, for every
	$\beta\in \left(\beta_s, \frac{n}{n-1}\right)$, every $t\in[0,T]$, and every
	positive function $w\in \mathcal{C}^\infty(M)$, we have
	\begin{align}\label{e1.10}
		&\int_Mw^s\left[
		|\nabla^2w|^2+\Ric_V(\nabla w,\nabla w)
		-w^{-1}|\nabla w|^2A_tw
		\right]d\mu_t \geq
	\widehat{\mathcal C}_t(\beta_s,\beta)
		\int_Mw^s|\nabla w|^2\,d\mu_t.
	\end{align}
\end{theorem}

\begin{remark}\label{rem:joint-modified-static-main}
	In the stationary reversible case of Remark \ref{r1.2}, Theorem \ref{main2} yields the weighted modified $\Gamma_2$ inequality with $\widehat{C}_{s, f}^{\text{mod}}(\beta):=\widehat{\Psi}_{\beta, \beta_s}\left(\lambda_{1, f}, L\right)+\beta \kappa.$
	If $f$ is constant, then
$
\widehat{\Psi}_{\beta, \beta_s}\left(\lambda_1, 0\right)=\left(1-\frac{n-1}{n} \beta\right) \lambda_1.
$
By passing to the limit $\beta \rightarrow \beta_s$, our estimate reduces to the Cai-Wang bound \cite[Theorem 3]{Cai} when $\rho=n-1$, throughout the interior of the admissible range. The case $s=2$ is treated separately (see Lemma \ref{l2.1}), while the endpoint $s=-\frac{2\left(2 n^2+1\right)}{4 n-1}$ is recovered from the unweighted argument, since all drift terms vanish.
\end{remark}

The two endpoint inequalities have the same sharp constant.
For each $t\in[0,T]$, define
\begin{equation}\label{eq:1.11}
	\ell_t:=
	\inf_{\substack{h\in\mathcal C^\infty(M)\\
			\nabla h\not\equiv0}}\frac{\displaystyle\int_M(A_th)^2\,d\mu_t-\int_M S_t(\nabla h,\nabla h)\,d\mu_t}{\displaystyle\int_M|\nabla h|^2\,d\mu_t}.
\end{equation}
The linear endpoint estimates corresponding to
$s=0$ in Theorem \ref{main1} and $s=2$ in Theorem \ref{main2} share the sharp constant $\ell_t$
defined in \eqref{eq:1.11}.
When $S_t\equiv0$, this constant equals $\lambda_1(t)$.
The endpoint inequalities and their sharpness are
established in Corollary \ref{c2.2}. 

We now turn to the evolution results and assume that
$\rho_t$ satisfies the Fokker-Planck equation
\eqref{q1.2}.
The following theorem gives a differential inequality
for the Tsallis entropy along the coupled forward-backward
evolution.
For $p\neq1,2$, the proof uses Theorem \ref{main2}.
At $p=1$, the result recovers the Shannon inequality
\eqref{e1.11}.
The quadratic endpoint $p=2$ follows from the sharp
estimate \eqref{p1.13}.

\begin{theorem}\label{main3}
Let $(M^n,g)$ be a closed, connected Riemannian manifold
of dimension $n\geq2$, equipped with a smooth vector field
$V\in\mathcal C^\infty(TM)$ and the family of probability
measures $(d\mu_t)_{t\in[0,T]}$ defined by \eqref{k1.3}.
Assume that the densities $\rho_t$ satisfy the
Fokker-Planck equation \eqref{q1.2} and that $u>0$
solves \eqref{eq:1.9} with  
$$\int_M u(\cdot,T)\,d\mu_T=1.$$
For 
$p\in\left(\frac{2(n-1)^2}{2n^2+1},\,2\right],$
define
$$\widehat\Lambda_{p,V}^*(t):=
\begin{cases}
	\ell_t,
	& p=2,\\[5pt]
		\displaystyle
	\sup_{\beta\in\left(\beta_p,\frac{n}{n-1}\right)}
	\widehat{\mathcal C}_t(\beta_p,\beta),
	& p<2,\\[5pt]
\end{cases}
\quad
\beta_p:=\frac{n-1}{n+2}
\left(\frac{2}{p}-1\right)
\quad (p<2).$$
At $p=1$, this coefficient reduces to
$\widehat\Lambda_{1,V}^*(t)
=\widehat C_{\mathrm{Sh},V}^*(t)$, where $\widehat C_{\mathrm{Sh},V}^*(t)
=
\sup_{\alpha\in\left(\frac{n-1}{n+2}, \frac{n}{n-1}\right)}
\widehat{\mathcal C}_t \left(\frac{n-1}{n+2},\alpha\right).$
With the convention $T_{1,V}=\mathcal S_V$,
the Tsallis  entropy $T_{p,V}$ satisfies
\begin{equation}\label{e1.7}
	\frac{d^2}{dt^2}T_{p,V}(u,t)
	-2\widehat\Lambda_{p,V}^*(t)
	\frac{d}{dt}T_{p,V}(u,t)
	\leq0, 	\quad t\in(0,T).
\end{equation}
\end{theorem}
\begin{remark}\label{rem:joint-time-orientation-main}
The sign in \eqref{e1.7} reflects the backward-time convention.	If $\tau=T-t$ and
	$$\widehat T_{p,V}(\tau):=T_{p,V}(u,T-\tau),$$ then
	\eqref{e1.7} becomes
	$$\frac{d^2}{d\tau^2}\widehat T_{p,V}(\tau)
	+2\widehat\Lambda_{p,V}^*(T-\tau)
	\frac{d}{d\tau}\widehat T_{p,V}(\tau)\leq0.$$
	In the stationary reversible case, this is the corresponding
	inequality along the forward Witten heat flow
	$\partial_\tau u=\Delta_fu.$
	The coefficient $p=2$ is $\lambda_{1,f}$ in that case, whereas
	for a general moving measure its exact value is $\ell_t$.
\end{remark}

The preceding entropy estimate gives a quantitative
propagation law for interpolation inequalities along
the Fokker-Planck evolution.
For $q\geq1$, a probability measure $\mu$, and a positive
smooth function $w$, define
\begin{equation}
\mathcal B_{q,\mu}(w):=
\begin{cases}
	\displaystyle
	\frac{\left(	\displaystyle\int_Mw^q\,d\mu\right)^{\frac{2}{q}}
		-	\displaystyle\int_Mw^2\,d\mu}{q-2},
	&q\neq2,\\
	\displaystyle\frac12\Ent_\mu(w^2),
	&q=2,
\end{cases}
\end{equation}
where 
$\Ent_\mu(w^2):=
\int_M w^2
\log\left(
\frac{w^2}{\int_M w^2\,d\mu}
\right)\,d\mu.$
\begin{corollary}\label{coro1.8}
	Let $(M^n,g)$ be a closed, connected Riemannian manifold
	of dimension $n\geq2$, equipped with a smooth vector field
	$V\in\mathcal C^\infty(TM)$ and the family of probability
	measures $(d\mu_t)_{t\in[0,T]}$ defined by \eqref{k1.3}.
	Assume that the densities $\rho_t$ satisfy the
	Fokker-Planck equation \eqref{q1.2}. 
	For 
	$q\in \left[1,\frac{2n^2+1}{(n-1)^2}\right),$
	set
	$$b_q:=\frac{(n-1)(q-1)}{n+2},
	\quad
	\widehat{\mathcal D}_{q,V}^*(t):=
	\begin{cases}
		\ell_t,
		& q=1,\\[3pt]
		\displaystyle
		\sup_{\beta\in\left(b_q,\frac{n}{n-1}\right)}
		\widehat{\mathcal C}_t(b_q,\beta),
		& 1<q<\frac{2n^2+1}{(n-1)^2}.
	\end{cases}$$
	Suppose that, for some $c_{q,0}>0$,
	\begin{equation}\label{k1.16}
		\int_M|\nabla w|^2\,d\mu_0
		\geq c_{q,0}\mathcal B_{q,\mu_0}(w)
	\end{equation}
	holds for all positive functions $w\in\mathcal C^\infty(M)$.
	Define $c_{q,T}$ by
	\begin{equation}\label{q1.19}
		\begin{aligned}
			c_{q,T}^{-1}
			={}&
			2\int_0^T
			\exp\left(
			-2\int_t^T\widehat{\mathcal D}_{q,V}^*(r)\,dr
			\right)\,dt+
			c_{q,0}^{-1}
			\exp\left(
			-2\int_0^T\widehat{\mathcal D}_{q,V}^*(r)\,dr
			\right).
		\end{aligned}
	\end{equation}
	Then,
	\begin{equation}\label{e1.8}
		\int_M|\nabla w|^2\,d\mu_T
		\geq c_{q,T}\mathcal B_{q,\mu_T}(w)
	\end{equation}
	for every positive function $w\in\mathcal C^\infty(M)$.
	In particular, $c_{q,T}>0$ for every finite $T$.
\end{corollary}
\begin{remark}
Set $p=\frac{2}{q}$ and solve \eqref{eq:1.9}
with terminal datum 
$u_T=\frac{w^q}{\int_M w^q\,d\mu_T}.$
With this normalization, the entropy bound
$I_p(T)\geq2c[-T_{p,V}(u,T)]$ becomes
$$\int_M|\nabla w|^2\,d\mu_T
\geq c\,\mathcal B_{q,\mu_T}(w),$$
where $T_{1,V}=\mathcal S_V$.
The initial entropy bound follows from
\eqref{k1.16} applied to
$u(\cdot,0)^{\frac{1}{q}}$. 
At $q=1$, choosing $c_{1,0}=\lambda_1(0)$ yields
the Poincar\'e bound $\lambda_1(T)\geq c_{1,T}$.
At $q=2$, we have
$\widehat{\mathcal D}_{2,V}^*
=\widehat C_{\mathrm{Sh},V}^*$,
and the Shannon estimate gives
\begin{equation}\label{eq:one-time-LSI-main}
	\Ent_{\mu_T}(w^2)
	\leq\frac{2}{c_{2,T}}
	\int_M|\nabla w|^2\,d\mu_T.
\end{equation}
\end{remark}

For the Witten Laplacian with its invariant probability
measure, the preceding entropy estimates yield the
following weighted interpolation inequalities.
\begin{corollary}\label{c1.10}
Let $(M^n,g)$ be a closed, connected Riemannian manifold
of dimension $n\geq2$, equipped with the probability measure 
$d\nu_f:=\frac{e^{-f}\,dv}{\int_Me^{-f}\,dv}, f\in\mathcal C^\infty(M).$
Assume that 
$$\Ric_f:=\Ric+\nabla^2f\geq\kappa g,
\quad
|\nabla f|\leq L$$
on $M$, for $\kappa\in\mathbb R$ and $L\geq0$.
	For $q\in \left(1,\frac{2n^2+1}{(n-1)^2}\right)$, set 
	$\widehat{\mathcal D}_{q,f}^*:=\sup_{\beta\in \left(b_q,\, \frac{n}{n-1}\right)}
	\left\{
	\widehat\Psi_{\beta,b_q}(\lambda_{1,f},L)
	+\beta\kappa
	\right\}.$
	Then, for $q\neq2$ and every positive function $w\in \mathcal{C}^\infty(M)$,
	\begin{align}\label{p1.22}
		\int_M|\nabla w|^2\,d\nu_f
		\geq\frac{\widehat{\mathcal D}_{q,f}^*}{q-2}
		\left[
		\left(\int_Mw^q\,d\nu_f\right)^{\frac{2}{q}}
		-\int_Mw^2\,d\nu_f
		\right].
	\end{align}
	At $q=1$, the corresponding inequality is the sharp weighted
	Poincar\'e inequality
	\begin{equation}\label{e1.6}
		\int_M|\nabla w|^2\,d\nu_f
		\geq\lambda_{1,f}\left[
		\int_Mw^2\,d\nu_f-
		\left(\int_Mw\,d\nu_f\right)^2
		\right].
	\end{equation}
	If $\widehat{\mathcal D}_{2,f}^*>0$, then
	\begin{equation}\label{q1.23}
		\Ent_{\nu_f}(w^2)
		\leq\frac{2}{\widehat{\mathcal D}_{2,f}^*}
		\int_M|\nabla w|^2\,d\nu_f.
	\end{equation}
\end{corollary}

The paper is organized as follows.
Section  \ref{s2} establishes the integral identities
and preliminary estimates used in the $\Gamma_{2,V}$
analysis. In Section  \ref{s3}, we prove the spectral
and modified spectral $\Gamma_2$ inequalities for the
$V$-Laplacian stated in Theorems  \ref{main1} and  \ref{main2}.
Section \ref{s4} develops the Shannon entropy estimates
and proves the Tsallis entropy inequality in
Theorem \ref{main3}. It then derives the interpolation
inequalities of Corollary \ref{coro1.8} and their
logarithmic Sobolev and stationary Witten consequences.
For completeness, the Appendix shows that, for every
smooth positive initial probability density, the
Fokker-Planck equation \eqref{q1.2} admits a unique
smooth solution that remains positive and preserves
total mass on every finite time interval.

\section{Integral identities and basis estimates}\label{s2}
\quad 
In this section, we establish the integral identities and estimates needed for the main results. We begin with an integrated Bochner-Weitzenb\"ock identity with respect to $d\mu_t$ at each fixed time. 
Together with the variational characterization of $\ell_t$ in
\eqref{eq:1.11}, this identity yields the sharp linear endpoint
estimates in Theorem \ref{main1} ($s=0$) and
Theorem \ref{main2} ($s=2$).

Let $(M^n,g)$ be a closed, connected Riemannian manifold
of dimension $n\geq2$, equipped with a smooth vector field
$V\in\mathcal C^\infty(TM)$ and the family of probability
measures $(d\mu_t)_{t\in[0,T]}$ defined by \eqref{k1.3}.
For each $t\in[0,T]$, write $d\mu_t=\rho_t\,dv$,
where $\rho_t$ is smooth and strictly positive, and define
\begin{equation}\label{o2.1}
	A_th
	:=
	\frac{1}{\rho_t}\operatorname{div}(\rho_t\nabla h)
	=
	\Delta h+\langle\nabla\log\rho_t,\nabla h\rangle,
	\quad h\in\mathcal C^\infty(M).
\end{equation}
Set 
$$W_t:=V-\nabla\log\rho_t,
\quad
S_t:=\frac12\mathcal L_{W_t}g,
\quad
\Ric_V:=\Ric-\frac12\mathcal L_Vg.$$ 
Fix $t\in[0,T]$. For $k=1,2$, set
$$H^k(M):=W^{k,2}(M)
=
\left\{
u\in L^2(M,dv)
\;\middle|\;
\int_M|\nabla^j u|^2\,dv<\infty,\,\,
 j=1,\ldots,k
\right\},$$
where the covariant derivatives are understood in the weak sense. 
Since $\rho_t$ is smooth and strictly positive on
the compact manifold $M$, replacing $dv$ by $d\mu_t$
gives the same spaces with equivalent norms. 
For real-valued functions $h,v\in\mathcal C^\infty(M)$,
the integration by parts shows that
\begin{align*}
	\langle -A_th,v\rangle_{L^2(M, d\mu_t)}
	&=-\int_Mv\,\operatorname{div}(\rho_t\nabla h)\,dv=\int_M\langle\nabla h,\nabla v\rangle\,d\mu_t
	=\langle h,-A_tv\rangle_{L^2(M, d\mu_t)}.
\end{align*}
In particular, taking $v=h$, we obtain 
$\langle -A_th,h\rangle_{L^2(M, d\mu_t)}
\geq0.$
Thus, $-A_t$ is symmetric and nonnegative on
$\mathcal C^\infty(M)$. By unitary equivalence and \cite[Theorem 14.7]{Zwo},
$-A_t$ is essentially self-adjoint in $L^2(M,d\mu_t)$.
Its nonnegative self-adjoint closure, still denoted by
$-A_t$, has domain $H^2(M)$ and compact resolvent.
From the energy identity and connectedness of $M$, we see that
$\ker(-A_t)=\operatorname{span}\{1\}$.
Therefore, we enumerate the eigenvalues as 
$0=\lambda_0(t)<\lambda_1(t)\leq\lambda_2(t)\leq\cdots.$
The Rayleigh-Ritz principle then yields
\begin{equation}\label{w2.6}
	\lambda_1(t)
	=
	\inf_{\substack{
			h\in H^1(M;\mathbb R)\setminus\{0\}\\
			\int_Mh\,d\mu_t=0}}
	\frac{\displaystyle\int_M|\nabla h|^2\,d\mu_t}
	{\displaystyle\int_Mh^2\,d\mu_t}.
\end{equation}
 Throughout this paper, we denote by $\Gamma_{2,V}$
the $\Gamma_2$ operator associated with the
$V$-Laplacian 
$$\Gamma_{2,V}(u)
:=
\frac{1}{2}\Delta_V|\nabla u|^2
-\left\langle\nabla u,\nabla\Delta_Vu\right\rangle
=
|\nabla^2u|^2+\Ric_V(\nabla u,\nabla u).$$
\begin{lemma}\label{l2.1}
	For each $t\in[0,T]$ and any
	$u\in\mathcal C^\infty(M)$, one has
	\begin{align}\label{h2.1}
		\int_M \Gamma_{2,V}(u)\,d\mu_t
		&=
		\int_M(A_tu)^2\,d\mu_t
		-\int_M S_t(\nabla u,\nabla u)\,d\mu_t.
	\end{align}
	Consequently,
	\begin{align}\label{e2.2}
		\int_M \Gamma_{2,V}(u)\,d\mu_t
		&\geq
		\lambda_1(t)\int_M|\nabla u|^2\,d\mu_t
		-\int_M S_t(\nabla u,\nabla u)\,d\mu_t.
	\end{align}
	Moreover, every positive $u\in\mathcal C^\infty(M)$
	satisfies
	\begin{align}\label{e2.3}
		&\int_M u^2\left[\Gamma_{2,V}(u)-\frac{|\nabla u|^2}{u} A_t u\right] d \mu_t
		\geq
		\lambda_1(t)\int_Mu^2|\nabla u|^2\,d\mu_t
		-\int_Mu^2S_t(\nabla u,\nabla u)\,d\mu_t.
	\end{align}
\end{lemma}
\begin{proof}[Proof of Lemma \ref{l2.1}]
	Fix $t\in[0,T]$ and set 
	$b_t:=\nabla\log\rho_t.$
	Since $\rho_t$ is smooth and strictly positive, $b_t$ is a
	smooth vector field on $M$. Moreover, 
	$A_t=\Delta+\langle b_t,\nabla\cdot\rangle := \Delta_{b_t}.$
	Applying the Bochner-Weitzenb\"ock formula
	\cite[Lemma 2.1]{Li1} to $A_t$ and using the fact that $S_t=\Ric_{b_t}-\Ric_V$, we get
\begin{align}
	\frac{1}{2} A_t|\nabla u|^2=\frac{1}{2} \Delta_{b_t}|\nabla u|^2 & =\left|\nabla^2 u\right|^2+\Ric_{b_t}(\nabla u, \nabla u)+\left\langle\nabla u, \nabla\left(\Delta_{b_t} u\right)\right\rangle \nonumber\\
	&=\left|\nabla^2 u\right|^2+\Ric_V(\nabla u, \nabla u)+\left\langle\nabla u, \nabla\left(A_t u\right)\right\rangle+S_t(\nabla u, \nabla u) .
\end{align}
	Using the above identity and
integrating with respect to $d\mu_t$, we find that 
\begin{align}\label{e2.5}
	\int_M \Gamma_{2,V}(u) d \mu_t
	& =\frac{1}{2} \int_M A_t|\nabla u|^2 d \mu_t-\int_M\left\langle\nabla u, \nabla\left(A_t u\right)\right\rangle d \mu_t-\int_M S_t(\nabla u, \nabla u) d \mu_t.
\end{align}
Since $M$ is a closed Riemannian manifold, the divergence theorem gives
$$
\int_M A_t|\nabla u|^2 d \mu_t=\int_M \operatorname{div}\left(\rho_t \nabla|\nabla u|^2\right) d v=0.
$$
Moreover,
$$
-\int_M\left\langle\nabla u, \nabla A_t u\right\rangle d \mu_t=-\int_M\left\langle\rho_t \nabla u, \nabla A_t u\right\rangle d v=\int_M\left(A_t u\right) \operatorname{div}\left(\rho_t \nabla u\right) d v=\int_M\left(A_t u\right)^2 d \mu_t.
$$
Substituting the two identities above into \eqref{e2.5}, we obtain \eqref{h2.1}.
	
	Since $\mu_t$ is a probability measure, the
	$\mu_t$-average of $u$ is 
	$\overline u_t:=\int_Mu\,d\mu_t.$
	Set $w:=u-\overline u_t$. Then, we have
$
\int_M w d \mu_t=0, \nabla w=\nabla u,$ and $A_t w=A_t u.
$
Thus, $w$ has zero mean with respect to the weighted measure
$d\mu_t$ and is therefore admissible in the variational
characterization of $\lambda_1(t)$. 
If $w\not\equiv0$, then \eqref{w2.6} deduces that 
$$\lambda_1(t)\int_Mw^2\,d\mu_t
\leq
\int_M|\nabla w|^2\,d\mu_t.$$
If $w\equiv0$, the same inequality holds trivially. Hence, we get 
	\begin{equation}\label{w2.5}
		\left(\int_Mw^2\,d\mu_t\right)^{\frac12}
		\leq
		\frac{1}{\sqrt{\lambda_1(t)}}
		\left(\int_M|\nabla u|^2\,d\mu_t\right)^{\frac12}.
	\end{equation}
This and 
	the Cauchy-Schwarz inequality entail that
	\begin{align*}
		\int_M|\nabla u|^2\,d\mu_t
		=
		-\int_MwA_tu\,d\mu_t
		&\leq
		\frac{1}{\sqrt{\lambda_1(t)}}
		\left(\int_M|\nabla u|^2\,d\mu_t\right)^{\frac12}
		\left(\int_M(A_tu)^2\,d\mu_t\right)^{\frac12}.
	\end{align*}
From this, we conclude that
	\begin{equation}\label{e2.6}
		\int_M(A_tu)^2\,d\mu_t
		\geq
		\lambda_1(t)\int_M|\nabla u|^2\,d\mu_t.
	\end{equation}
	This, combined with \eqref{h2.1} implies \eqref{e2.2}.
	
	We next turn to the proof of \eqref{e2.3}. To this end, we employ the change of variables $\sigma=u^2$ to rewrite the right-hand side of \eqref{e2.3} in the form of the right-hand side of \eqref{e2.2}. 
	Let $\left\{e_i\right\}_{i=1}^n$ be a local orthonormal frame. We denote the components of $\nabla u$ with respect to $\left\{e_i\right\}_{i=1}^n$ by $u_i:=d u\left(e_i\right)$. Then, we have $(d u \otimes d u)\left(e_i, e_j\right)=u_i u_j.$ Hence,
	\begin{equation}\label{e2.9}
	\left\langle\nabla^2 u, d u \otimes d u\right\rangle=(\nabla^2u)(\nabla u,\nabla u),\quad 	|d u \otimes d u|^2=|\nabla u|^4.
	\end{equation}
From \eqref{e2.9}, we obtain
$$
\frac{1}{4}\left|\nabla^2 \sigma\right|^2= u^2\left|\nabla^2 u\right|^2+2 u\left\langle\nabla^2 u, d u \otimes d u\right\rangle+|d u \otimes d u|^2=u^2\left|\nabla^2 u\right|^2+2 u\left(\nabla^2 u\right)(\nabla u, \nabla u)+|\nabla u|^4 .
$$
This and the identity 	$\Ric_V(\nabla\sigma,\nabla\sigma)
=
4u^2\Ric_V(\nabla u,\nabla u),$ give that
\begin{align}\label{e2.12}
	 \frac{1}{4} 	\int_M \Gamma_{2,V}(\sigma)\,d\mu_t
	=\int_M\left[u^2\left|\nabla^2 u\right|^2+u^2 \Ric_V(\nabla u, \nabla u)+2 u\left(\nabla^2 u\right)(\nabla u, \nabla u)+|\nabla u|^4\right] d \mu_t.
\end{align}
Using the integration by parts formula, we see that 
	\begin{equation}\label{e2.11} 
	-\int_M u|\nabla u|^2 A_t u d \mu_t=-\int_M u|\nabla u|^2 \operatorname{div}\left(\rho_t \nabla u\right) d v=\int_M\left\langle\nabla\left(u|\nabla u|^2\right), \nabla u\right\rangle d \mu_t.
	\end{equation}
Note that 
\begin{align}\label{q2.16}
\left\langle \nabla |\nabla u|^2,\nabla u\right\rangle=(\nabla u)\bigl(|\nabla u|^2\bigr)=2\left\langle \nabla_{\nabla u}\nabla u,\nabla u\right\rangle=2(\nabla^2u)(\nabla u,\nabla u).
\end{align}
This, \eqref{e2.12} and \eqref{e2.11} lead to
\begin{align}\label{w2.14}
\int_M u^2\left[\Gamma_{2,V}(u)-\frac{|\nabla u|^2}{u} A_t u\right] d \mu_t=\frac{1}{4} 	\int_M \Gamma_{2,V}(\sigma)\,d\mu_t.
\end{align}
	Together with \eqref{e2.6}, this yields \eqref{e2.3}. \end{proof}

Lemma \ref{l2.1} and the variational definition of
$\ell_t$ in \eqref{eq:1.11} give the following sharp estimates at
the linear endpoints $s=0$ and $s=2$. Indeed, \eqref{e2.2} and a test with an eigenfunction
of $-A_t$ corresponding to $\lambda_1(t)$ give
$\ell_t\in\mathbb R$, with $\ell_t=\lambda_1(t)$ if $S_t=0$.
Any test function can be made positive by adding a constant
without changing the quotient in \eqref{eq:1.11}, so
\eqref{p1.13} is sharp for positive functions. Finally, applying \eqref{p1.13} to $h=w^2$ and
using the identity \eqref{w2.14}, we obtain \eqref{p1.14}
with the same sharp constant $\ell_t$.
\begin{corollary}\label{c2.2}
	Let $(M^n,g)$ be a closed, connected Riemannian manifold
	of dimension $n\geq2$, equipped with a smooth vector
	field $V\in\mathcal C^\infty(TM)$ and the family of
	probability measures $(d\mu_t)_{t\in[0,T]}$ defined
	by \eqref{k1.3}.
	For each $t\in[0,T]$, let $\ell_t$ be defined
	by \eqref{eq:1.11}. Then, $\ell_t\in\mathbb R$, and
	every positive $w\in\mathcal C^\infty(M)$ satisfies
	\begin{align}
		\int_M\Gamma_{2,V}(w)\,d\mu_t
		&\geq
		\ell_t\int_M|\nabla w|^2\,d\mu_t,
		\label{p1.13}\\
		\int_M w^2\left[
		\Gamma_{2,V}(w)
		-w^{-1}|\nabla w|^2A_tw
		\right]\,d\mu_t
		&\geq
		\ell_t\int_Mw^2|\nabla w|^2\,d\mu_t.
		\label{p1.14}
	\end{align}
	At each fixed time $t$, $\ell_t$ is the largest
	constant in each inequality valid for all positive
	smooth functions.
	Moreover, if $S_t\equiv0$, then
	$\ell_t=\lambda_1(t)$.
\end{corollary}
We next establish two integral identities of Cai-Wang type \cite[Lemma 2.2]{Cai}, which will be used in the proofs of Theorems \ref{main1} and \ref{main2}.
\begin{lemma}\label{l2.2}
	For each $t\in[0,T]$ and a positive function
	$u\in\mathcal C^\infty(M)$, set
$$
\widetilde{H}_t u:=\nabla^2 u+d\left(\log \rho_t\right) \odot d u-\frac{A_t u}{n} g,
$$
	where 
	$$\eta\odot\zeta
	:=
	\frac{1}{2}
	\left(\eta\otimes\zeta+\zeta\otimes\eta\right)$$
	for one-forms $\eta$ and $\zeta$, and put
	\begin{equation}\label{w3.1}
		Q(u):=
		\frac{du\otimes du}{u}
		-\frac{|\nabla u|^2}{nu}g.
	\end{equation}
	Then, we have
	\begin{align}\label{e2.14}
		\int_M\frac{|\nabla u|^2A_tu}{u}\,d\mu_t
		&=
		\frac{n}{n+2}
		\int_M\frac{|\nabla u|^4}{u^2}\,d\mu_t
		-\frac{2n}{n+2}
		\int_M\langle\widetilde H_tu,Q(u)\rangle\,d\mu_t
		\nonumber\\
		&\quad
		+\frac{2n}{n+2}
		\int_M
		\frac{|\nabla u|^2
			\langle\nabla\log\rho_t,\nabla u\rangle}{u}
		\,d\mu_t.
	\end{align}
	Moreover,
	\begin{align}\label{l2.15}
		\int_M
		\left\langle\nabla^2u,\frac{du\otimes du}{u}\right\rangle
		\,d\mu_t
		&=
		\frac{1}{n+2}
		\int_M\frac{|\nabla u|^4}{u^2}\,d\mu_t
		+\frac{n}{n+2}
		\int_M\langle\widetilde H_tu,Q(u)\rangle\,d\mu_t
		\nonumber\\
		&\quad
		-\frac{n}{n+2}
		\int_M
		\frac{|\nabla u|^2
			\langle\nabla\log\rho_t,\nabla u\rangle}{u}
		\,d\mu_t.
	\end{align}
\end{lemma}

\begin{proof}[Proof of Lemma \ref{l2.2}]
We now introduce the quantities
$$
I:=\int_M \frac{|\nabla u|^2 A_t u}{u} d \mu_t, \quad J:=\int_M\left\langle\nabla^2 u, \frac{d u \otimes d u}{u}\right\rangle d \mu_t .
$$	Using the integration by parts formula, we obtain
	\begin{align}\label{e2.15}
	I=-\int_M\left\langle\nabla\left(\frac{|\nabla u|^2}{u}\right), \nabla u\right\rangle d \mu_t=\int_M \frac{|\nabla u|^4}{u^2}d \mu_t -\int_M\frac{1}{u}\left\langle\nabla |\nabla u|^2, \nabla u\right\rangle d \mu_t.
	\end{align}
Plugging \eqref{e2.9} and \eqref{q2.16} in the above quantity,  we deduce that
\begin{align}\label{e2.17}
	I=\int_M \frac{|\nabla u|^4}{u^2} d\mu_t-2 \int_M\left\langle\nabla^2 u, \frac{d u \otimes d u}{u}\right\rangle d\mu_t=\int_M \frac{|\nabla u|^4}{u^2} d\mu_t-2J.
\end{align}
	We next express $J$ in terms of $\widetilde H_tu$.
	By definition of $\widetilde H_tu$, one finds  that
	\begin{equation}\label{e2.22}
	J=\int_M J_1 d \mu_t-\int_M J_2 d \mu_t+\int_M J_3 d \mu_t,
	\end{equation}
	where
	\vspace{-0.1cm}
$$
J_1:=\left\langle\widetilde{H}_t u, \frac{d u \otimes d u}{u}\right\rangle, \quad J_2:=\left\langle d\left(\log \rho_t\right) \odot d u, \frac{d u \otimes d u}{u}\right\rangle,\quad J_3:=
\left\langle
\frac{A_tu}{n}g,\frac{du\otimes du}{u}
\right\rangle.
$$
We next make use of the identity $\langle A, g\rangle=\tr_g A$, which holds for every $(0,2)$-tensor $A,$ in order to establish \eqref{l2.15}. Note that
	\begin{align*}
		\langle\widetilde H_tu,g\rangle
		&=
		\operatorname{tr}_g
		\left(
		\nabla^2u+d(\log\rho_t)\odot du
		-\frac{A_tu}{n}g
		\right)=
		\Delta u+\langle\nabla\log\rho_t,\nabla u\rangle-A_tu
		=0.
	\end{align*}
	Thus,  $\widetilde H_tu$ is trace-free and
	\begin{equation}\label{e2.23}
	J_1=\left\langle\widetilde{H}_t u, \frac{d u \otimes d u}{u}-\frac{|\nabla u|^2}{n u} g\right\rangle=\left\langle\widetilde{H}_t u, Q(u)\right\rangle.
	\end{equation}
	To compute $J_2$, we use the identity 
	$\langle\eta\otimes\zeta,\theta\otimes\omega\rangle
	=
	\langle\eta,\theta\rangle
	\langle\zeta,\omega\rangle$
	for one-forms $\eta,\zeta,\theta,\omega$.
	It follows that
	\begin{align}\label{e2.24}
		J_2
		&=
		\frac{1}{2u}
		\left\langle
		d(\log\rho_t)\otimes du,du\otimes du
		\right\rangle
		+
		\frac{1}{2u}
		\left\langle
		du\otimes d(\log\rho_t),du\otimes du
		\right\rangle=
		\frac{|\nabla u|^2
			\langle\nabla\log\rho_t,\nabla u\rangle}{u}.
	\end{align}
	Similarly,
\begin{align}
J_3=\frac{A_t u}{n u}\langle g, d u \otimes d u\rangle=\frac{A_t u}{n u} \operatorname{tr}_g(d u \otimes d u)=\frac{|\nabla u|^2 A_t u}{n u}.\nonumber
\end{align}
From this,  \eqref{e2.22}, \eqref{e2.23} and \eqref{e2.24}, we deduce that
\begin{align}\label{k2.24}
J=\int_M\left\langle\widetilde{H}_t u, Q(u)\right\rangle d \mu_t-\int_M \frac{|\nabla u|^2\left\langle\nabla \log \rho_t, \nabla u\right\rangle}{u} d \mu_t+\frac{I}{n}
\end{align}
	Substituting this expression into \eqref{e2.17} yields
$$
\frac{n+2}{n} I=\int_M \frac{|\nabla u|^4}{u^2} d \mu_t-2 \int_M\left\langle\widetilde{H}_t u, Q(u)\right\rangle d \mu_t+2 \int_M \frac{|\nabla u|^2\left\langle\nabla \log \rho_t, \nabla u\right\rangle}{u} d \mu_t.
$$
Thus, \eqref{e2.14} is established. Applying \eqref{e2.14} to the term $\frac{1}{n}I$ in \eqref{k2.24} immediately yields \eqref{l2.15}.
\end{proof}
 \section{Spectral $\Gamma_2$ estimates for the $V$-Laplacian}
\label{s3}

\quad In this section, we prove Theorems  \ref{main1}
and  \ref{main2} by combining the weighted integral
identities and estimates established in Section  \ref{s2}
with the spectral gap $\lambda_1(t)$ of $-A_t$.
Throughout the proofs, $t\in[0,T]$ is fixed. 
Following the strategy of Cai-Wang \cite{Cai},
we use a power substitution and complete squares
involving the weighted trace-free Hessian
$\widetilde H_tu$ and the tensor $Q(u)$ defined
in \eqref{w3.1}.
The contribution of the general drift is incorporated
through $\mathcal R_{\alpha,t}$.
The terms containing $\nabla\log\rho_t$ are estimated
together with the spectral terms, leading to the
coefficient $\widehat{\mathcal C}_t$.

\subsection{Proof of Theorem  \ref{main1}}

\begin{proof}[Proof of Theorem  \ref{main1}]
	Fix $t\in[0,T]$ and
	$\alpha\in\left(\alpha_s,\frac{n}{n-1}\right)$.
	We first consider $s\neq-2$ and introduce
$u:=w^{\frac{s+2}{2}}, r:=-\frac{s}{s+2}.$
	Then, we get $w=u^{r+1},  w^s=u^{-2 r},$ and
$$
\nabla w=(r+1) u^r \nabla u, \quad \nabla^2 w=(r+1) u^r\left(\nabla^2 u+r \frac{d u \otimes d u}{u}\right).
$$
	Consequently,
$$
\int_M w^s|\nabla w|^2 d \mu_t=(r+1)^2 \int_M|\nabla u|^2 d \mu_t,
$$
	and
$$
\int_M w^s \Gamma_{2, V}(w) d \mu_t=(r+1)^2 \int_M\left[\left|\nabla^2 u+r \frac{d u \otimes d u}{u}\right|^2+\Ric_V(\nabla u, \nabla u)\right] d \mu_t.
$$
Thus, to prove Theorem \ref{main1}, it is enough to estimate the integral
$$
\mathcal{I}:=\int_M\left[\left|\nabla^2 u+r \frac{d u \otimes d u}{u}\right|^2+\Ric_V(\nabla u, \nabla u)\right] d \mu_t.
$$
For the sake of convenience in the subsequent proof, we now set
$$
\mathcal{A}:=\int_M \frac{|\nabla u|^4}{u^2} d\mu_t, \quad E:=\int_M|\nabla u|^2 d\mu_t, \quad 	F:=\langle\nabla\log\rho_t,\nabla u\rangle, \quad 	X:=A_tu.
$$
	We also write
$$
J:=\int_M\left\langle\nabla^2 u, \frac{d u \otimes d u}{u}\right\rangle d \mu_t, \quad D:=\int_M \frac{|\nabla u|^2 F}{u} d \mu_t .
$$
If $E=0$, then $u$, and hence, $w$ is constant, and the
inequality \eqref{e1.4} is immediate. We therefore assume $E>0$. Note that
\begin{align}\label{q3.4}
	\left|\nabla^2 u+r \frac{d u \otimes d u}{u}\right|^2 =\left|\nabla^2 u\right|^2+2 r\left\langle\nabla^2 u, \frac{d u \otimes d u}{u}\right\rangle+r^2 \frac{|\nabla u|^4}{u^2} .
\end{align}
Consequently,
\begin{align}
\mathcal{I}=\int_M\left[\left|\nabla^2 u\right|^2+\Ric_V(\nabla u, \nabla u)\right] d \mu_t+2 r J+r^2 \mathcal{A}.
\end{align}
	By \eqref{h2.1}, we see that 
$$
\int_M\left[\left|\nabla^2 u\right|^2+\Ric_V(\nabla u, \nabla u)\right] d \mu_t=\int_M X^2 d \mu_t-\int_M S_t(\nabla u, \nabla u) d \mu_t.
$$
	Splitting this identity with the parameter $\alpha$,
	we obtain
$$
\begin{aligned}
	\mathcal{I}= & (1-\alpha) \int_M X^2 d \mu_t+\alpha \int_M\left|\nabla^2 u\right|^2 d \mu_t +\int_M\left[\alpha \Ric_V+(\alpha-1) S_t\right](\nabla u, \nabla u) d \mu_t+2 r J+r^2 \mathcal{A}.
\end{aligned}
$$
	Since 
	$\alpha\Ric_V+(\alpha-1)S_t
	=\mathcal R_{\alpha,t}
	\geq\mathfrak r_\alpha(t)g,$
	it follows that
	\begin{align}\label{e3.5}
		\mathcal I
		&\geq
		\mathfrak r_\alpha(t)E
		+
		\int_M
		\left[
		(1-\alpha)X^2+\alpha|\nabla^2u|^2
		\right]d\mu_t
		+2rJ+r^2\mathcal A,
	\end{align}
where $\mathfrak{r}_\alpha(t):=\inf _{\substack{x \in M \\|\xi|_g=1}}\left(\mathcal{R}_{\alpha, t}\right)_x(\xi, \xi)$.
	We next separate the weighted trace of the Hessian.
A straightforward computation yields
\begin{align}\label{r3.7}
	(1-\alpha) X^2+\alpha\left|\nabla^2 u\right|^2 =\frac{d_\alpha}{n} X^2-\frac{2 \alpha}{n} X F+\alpha\left|\nabla^2 u-\frac{X}{n} g\right|^2,
\end{align}
where 
$d_\alpha:=n-(n-1)\alpha.$
Substituting this into \eqref{e3.5} gives 
	\begin{align}\label{e3.7}
	\mathcal{I} \geq \mathfrak{r}_\alpha(t) E+\int_M\left(\frac{d_\alpha}{n} X^2-\frac{2 \alpha}{n} X F
\right) d \mu_t+\alpha \int_M\left|\nabla^2 u-\frac{X}{n} g\right|^2 d \mu_t+2 r J+r^2 \mathcal{A}.
	\end{align}
	The identity \eqref{l2.15} yields
$$
J=\frac{1}{n+2} \mathcal{A}+\frac{n}{n+2} \int_M\left\langle\widetilde{H}_t u, Q(u)\right\rangle d \mu_t-\frac{n}{n+2} D,
$$
	where
$\widetilde{H}_t u=\nabla^2 u+d\left(\log \rho_t\right) \odot d u-\frac{X}{n} g$.
From this and \eqref{e3.7}, we have
	\begin{align}\label{e3.6}
		\mathcal{I} &\geq \mathfrak{r}_\alpha(t) E+\int_M\left(\frac{d_\alpha}{n} X^2-\frac{2 \alpha}{n} X F
\right) d \mu_t+\alpha \int_M\left|\nabla^2 u-\frac{X}{n} g\right|^2 d \mu_t \nonumber\\
		&\quad +\frac{2 n r}{n+2} \int_M\left\langle\widetilde{H}_t u, Q(u)\right\rangle d \mu_t+\left(r^2+\frac{2 r}{n+2}\right) \mathcal{A}-\frac{2 n r}{n+2} D.
	\end{align}
We now define 
	$$Z:=
	\nabla^2u-\frac{X}{n}g
	+\frac{nr}{(n+2)\alpha}Q(u).$$
The fact that 
$\nabla^2 u-\frac{X}{n} g=\widetilde{H}_t u-d\left(\log \rho_t\right) \odot d u$ deduces that
	\begin{align}\label{k3.9}
		|Z|^2
		&=
		\left|\nabla^2u-\frac{X}{n}g\right|^2
		+
		\frac{2nr}{(n+2)\alpha}
		\langle\widetilde H_tu,Q(u)\rangle
		\nonumber\\
		&\quad
		-
		\frac{2nr}{(n+2)\alpha}
		\langle d(\log\rho_t)\odot du,Q(u)\rangle
		+
		\frac{n^2r^2}{(n+2)^2\alpha^2}|Q(u)|^2.
	\end{align}
	The definition of $Q(u)$ gives
	\begin{align}\label{p3.11}
		\langle d(\log\rho_t)\odot du,Q(u)\rangle
		&=
		\frac{|\nabla u|^2F}{u}
		-
		\frac{|\nabla u|^2}{nu}
		\operatorname{tr}_g
		\bigl(d(\log\rho_t)\odot du\bigr)=
		\frac{n-1}{n}\frac{|\nabla u|^2F}{u}.
	\end{align}
	Moreover,
	\begin{align}\label{w3.9}
		|Q(u)|^2
		&=
		\frac{|du\otimes du|^2}{u^2}
		-
		\frac{2|\nabla u|^2}{nu^2}
		\langle du\otimes du,g\rangle
		+
		\frac{|\nabla u|^4}{n^2u^2}|g|^2
		=
		\frac{n-1}{n}\frac{|\nabla u|^4}{u^2}.
	\end{align}
This, \eqref{k3.9} and \eqref{p3.11} give that
	\begin{align*}
		\alpha\int_M|Z|^2\,d\mu_t
		&=
		\alpha\int_M
		\left|\nabla^2u-\frac{X}{n}g\right|^2d\mu_t
		+
		\frac{2nr}{n+2}
		\int_M\langle\widetilde H_tu,Q(u)\rangle\,d\mu_t
		-
		\frac{2r(n-1)}{n+2}D
		+
		\frac{n(n-1)r^2}{(n+2)^2\alpha}\mathcal A.
	\end{align*}
This result, combined with \eqref{e3.6}, yield that
	\begin{align}\label{e3.10}
	\mathcal{I} \geq \mathfrak{r}_\alpha(t) E+\int_M\left(\frac{d_\alpha}{n} X^2-\frac{2 \alpha}{n} X F
\right) d \mu_t+\alpha \int_M|Z|^2 d \mu_t+K_{s, \alpha} \mathcal{A}-\frac{2 r}{n+2} D,
	\end{align}
	where 
	$K_{s,\alpha}
	:=r^2+\frac{2r}{n+2}
	-\frac{n(n-1)r^2}{(n+2)^2\alpha}.$ Notice that
	$
	\tr_g Z=\Delta u-A_t u=-F
	.$
	Consequently,
	\[
	|Z|^2
	=
	\left|Z+\frac{F}{n}g\right|^2+\frac{F^2}{n}
	\geq\frac{F^2}{n}.
	\]
	Combining this with \eqref{e3.10}, we arrive at
	\begin{align}\label{e3.11}
	\mathcal{I} \geq \mathfrak{r}_\alpha(t) E+\int_M\left(\frac{d_\alpha}{n} X^2-\frac{2 \alpha}{n} X F+\frac{\alpha}{n} F^2\right) d \mu_t+K_{s, \alpha} \mathcal{A}-\frac{2 r}{n+2} D .
	\end{align}
Since 
	$r=-\frac{s}{s+2}, 
	\alpha_s=
	\frac{n(n-1)s}{(n+2)(ns-4)},$
	we get
	$\frac{n(n-1)r^2}{(n+2)^2}\alpha_s^{-1}
	=
	r^2+\frac{2r}{n+2}.$
	It follows that
$$
K_{s, \alpha}=\frac{n(n-1) r^2}{(n+2)^2}\left(\alpha_s^{-1}-\alpha^{-1}\right)>0,
$$
	since $\alpha>\alpha_s>0$ and $r\neq0$. 
We next estimate the quadratic expression involving $X$ and $F$.
Set 
$$x:=\frac{\|X\|_{L^2\left(M, d \mu_t\right)}}{\sqrt{E}}, \quad y:=\frac{\|F\|_{L^2\left(M, d \mu_t\right)}}{\sqrt{E}} .$$
	The spectral estimate \eqref{e2.6} and the fact that $|F| \leq H_t|\nabla u|$ give $x\geq \sqrt{\lambda_1(t)}$ and $0\leq y \leq H_t.$ 
From the Cauchy-Schwarz inequality, we see that
	\[
	\left|\int_MXF\,d\mu_t\right|
	\leq
	\|X\|_{L^2(M, d\mu_t)}
	\|F\|_{L^2(M, d\mu_t)}
	=
	xyE.
	\]
	Since $\alpha>0$, this yields
\begin{align}\label{w3.11}
\int_M\left(\frac{d_\alpha}{n} X^2-\frac{2 \alpha}{n} X F+\frac{\alpha}{n} F^2\right) d \mu_t \geq\left(\frac{d_\alpha}{n} x^2-\frac{2 \alpha}{n} x y+\frac{\alpha}{n} y^2\right) E.
\end{align}
On the other hand, by the Cauchy-Schwarz inequality, we have
	$$
	|D|=\left|\int_M \frac{|\nabla u|^2 F}{u} d \mu_t\right| \leq\left(\int_M \frac{|\nabla u|^4}{u^2} d \mu_t\right)^{\frac{1}{2}}\left(\int_M F^2 d \mu_t\right)^{\frac{1}{2}}=y \sqrt{\mathcal{A} E}.
	$$
	 Therefore,
	\begin{align}\label{w3.12}
		K_{s,\alpha}\mathcal A-\frac{2r}{n+2}D
		&\geq
	K_{s, \alpha}\left(\sqrt{\mathcal{A}}-\frac{|r| y}{(n+2) K_{s, \alpha}} \sqrt{E}\right)^2-\frac{r^2 y^2}{(n+2)^2 K_{s, \alpha}} E \geq-\chi_{\alpha_s, \alpha} y^2 E,
	\end{align}
where $\chi_{a, \alpha}:=\frac{1}{n(n-1)\left(a^{-1}-\alpha^{-1}\right)}$.
	Combining \eqref{w3.11} and \eqref{w3.12} with \eqref{e3.11}, we obtain
	\begin{align}\label{e3.12}
		\mathcal I
		&\geq
		\left[
		\mathfrak r_\alpha(t)
		+
		\widehat\Psi_{\alpha,\alpha_s}
		\bigl(\lambda_1(t),H_t\bigr)
		\right]E=
		\widehat{\mathcal C}_t(\alpha_s,\alpha)E.
	\end{align}
	Here, the above second inequality follows from
	$x\geq\sqrt{\lambda_1(t)}$ and $0\leq y\leq H_t$.
	Returning to $w$, we conclude that
	\begin{align*}
		\int_Mw^s\Gamma_{2,V}(w)\,d\mu_t
		=
		(r+1)^2\mathcal I\geq
		\widehat{\mathcal C}_t(\alpha_s,\alpha)(r+1)^2E=
		\widehat{\mathcal C}_t(\alpha_s,\alpha)
		\int_Mw^s|\nabla w|^2\,d\mu_t.
	\end{align*}
	This proves \eqref{e1.4} for $s\neq-2$.
	
	We now consider $s=-2$, where the preceding change
	of variables degenerates. Fix
	$\alpha_{-2}<\alpha<\frac{n}{n-1}$, where
	$\alpha_{-2}=\frac{n(n-1)}{(n+2)^2}$.
	By continuity of $s\mapsto\alpha_s$, we have
	$0<\alpha_s<\alpha$ for all $s$ sufficiently close
	to $-2$, so the estimate already proved applies
	to such $s\neq-2$.
	Moreover, since $0\leq y\leq H_t$ in the joint
	minimization,
$$
\left|\widehat{\mathcal{C}}_t\left(\alpha_s, \alpha\right)-\widehat{\mathcal{C}}_t\left(\alpha_{-2}, \alpha\right)\right| \leq H_t^2\left|\chi_{\alpha_s, \alpha}-\chi_{\alpha_{-2}, \alpha}\right| \longrightarrow 0 \,\, \text {as}\,\, s \rightarrow-2 .
$$
	Here the convergence follows from continuity of
	$a\mapsto\chi_{a,\alpha}$ at $a=\alpha_{-2}<\alpha$.
	Positivity of $w$ and compactness of $M$ also give
	$w^s\longrightarrow w^{-2}$ uniformly on $M$.
	Passing to the limit in the already established
	inequality therefore proves \eqref{e1.4} for $s=-2$. 
Finally, for each fixed admissible $s$, the factor
$\int_M w^s|\nabla w|^2\,d\mu_t$ is nonnegative
and independent of $\alpha$. Taking the supremum
over the admissible values of $\alpha$ therefore
gives the optimized estimate.
\end{proof}

We next evaluate the joint minimum in \eqref{eq:1.3} explicitly and determine when it strictly improves the bound obtained by estimating the correction term $-\chi_{a,\alpha}y^2$ separately.
\begin{proposition}\label{p1.3}
	For
	$
	0<a<\alpha<\frac{n}{n-1}, \lambda>0,$ and $L \geq 0,$ set
	$$
	l:=\sqrt{\lambda}, \quad A_0:=\frac{d_\alpha}{n}, \quad B_0:=\frac{\alpha}{n}, \quad C_0:=\frac{\alpha}{n}-\chi_{a, \alpha},
	$$
	and
	$
	\Delta_0:=C_0-\frac{B_0^2}{A_0}=\frac{\alpha(1-\alpha)}{d_\alpha}-\chi_{a, \alpha}. 
	$
	A minimizing pair in \eqref{eq:1.3} is
\begin{align}\label{o3.14}
\left(x_*, y_*\right)= \begin{cases}\left(l, \min \left\{L, \frac{B_0 l}{C_0}\right\}\right), & \Delta_0 \geq 0, \vspace{0.3cm}\\
	\left(\max \left\{l, \frac{B_0 L}{A_0}\right\}, L\right), & \Delta_0<0 .\end{cases}
\end{align}
	In the case $\Delta_0 \geq 0$, we have $C_0>0$, so the
	expression defining $y_*$ is well-defined. Moreover, in both
	cases $\Delta_0 \geq 0$ and $\Delta_0<0$, the minimum value
	$\widehat{\Psi}_{\alpha,a}(\lambda,L)$ is attained at the
	corresponding pair $(x_*,y_*)$ and is given by 
	$\widehat{\Psi}_{\alpha,a}(\lambda,L)
	=A_0x_*^2-2B_0x_*y_*+C_0y_*^2.$
	We now introduce the quantity
	$$
	\Psi_\alpha(\lambda, L):=\min _{\substack{x \geq \sqrt{\lambda} \\ 0 \leq y \leq L}}\left\{\frac{d_\alpha}{n} x^2-\frac{2 \alpha}{n} x y+\frac{\alpha}{n} y^2\right\}.
	$$
	Then, we obtain
	\begin{equation}\label{t1.6}
		\widehat\Psi_{\alpha,a}(\lambda,L)
		\geq\Psi_\alpha(\lambda,L)-\chi_{a,\alpha}L^2.
	\end{equation}
	The inequality is strict precisely when $\alpha<1$ and
	$L>\sqrt\lambda$.
\end{proposition}
Proposition \ref{p1.3} follows by minimizing the quadratic form first
in $x$ and then in $y$, and observing that equality in
\eqref{t1.6} holds precisely when $\Psi_\alpha(\lambda,L)$
admits a minimizer with $y=L$.
\subsection{Proof of Theorem \ref{main2}}

\quad\quad In this subsection, we prove the modified Bakry-\'{E}mery $\Gamma_{2, V}$ inequality \eqref{e1.10} of Theorem \ref{main2} by adapting the argument used for Theorem \ref{main1}.

\begin{proof}[Proof of Theorem \ref{main2}]
	Fix $t\in[0,T]$ and
	$s\in\left(-\infty,-\frac{2(2n^2+1)}{4n-1}\right)\cup(2,\infty).$ Observe that
	 $\beta_s \in\left(0, \frac{n}{n-1}\right).$ Set 
	$q:=\frac{2}{s+2},
	r:=q-1=-\frac{s}{s+2},$ and
	$\tau:=\frac{s-2}{s+2}.$
	Notice that $\tau=-(2r+1).$ For $s \in\left(-\infty,-\frac{2\left(2 n^2+1\right)}{4 n-1}\right) \cup(2, \infty),$ one has $\tau>0$.
	Let $u:=w^{\frac{s+2}{2}},$
	so that $w=u^q=u^{r+1}.$
	Since 
	$qs=\frac{2s}{s+2}=-2r,$ 
	we have 
	$w^s=u^{-2r}.$
	Notice that $\nabla w=q u^r \nabla u$ and
	$$
	\nabla^2 w=q u^r\left(\nabla^2 u+r \frac{d u \otimes d u}{u}\right),\quad A_t w=q u^r\left(A_t u+r \frac{|\nabla u|^2}{u}\right).
	$$
	These identities, combined with the fact that $w^s=u^{-2r}$, yield that 
	\begin{align}\label{e3.25}
		\int_M w^s|\nabla w|^2 d\mu_t=q^2 \int_M|\nabla u|^2 d\mu_t,
	\end{align}
	and
	\begin{align}\label{w3.19}
		&\int_M w^s\left[\Gamma_{2, V}(w)-w^{-1}|\nabla w|^2 A_t w\right] d \mu_t
		\nonumber\\
		&\quad\quad=
		q^2 \int_M\left[\left|\nabla^2 u+r \frac{d u \otimes d u}{u}\right|^2+\Ric_V(\nabla u, \nabla u)-q \frac{|\nabla u|^2}{u}\left(A_t u+r \frac{|\nabla u|^2}{u}\right)\right] d\mu_t .
	\end{align}
	From this, \eqref{e2.17} and \eqref{q3.4}, one finds that
	$$\begin{aligned}
		\int_M w^s\left[\Gamma_{2, V}(w)-w^{-1}|\nabla w|^2 A_t w\right] d \mu_t  =q^2\left[\int_M\left[\left|\nabla^2 u\right|^2+\Ric_V(\nabla u, \nabla u)\right] d\mu_t+\tau \int_M \frac{|\nabla u|^2 A_t u}{u} d\mu_t\right].
	\end{aligned}
	$$
	From this and \eqref{e3.25}, we see that to prove Theorem \ref{main2}, it is enough to estimate the integral
	$$
	\mathcal{J}=\int_M\left[\left|\nabla^2 u\right|^2+\Ric_V(\nabla u, \nabla u)\right] d\mu_t+\tau \int_M \frac{|\nabla u|^2 A_t u}{u} d \mu_t.
	$$
Following the same approach as in the proof of Theorem \ref{main1}, we introduce the following notation to streamline the subsequent argument
$$
\mathcal{A}:=\int_M \frac{|\nabla u|^4}{u^2} d\mu_t, \quad E:=\int_M|\nabla u|^2 d\mu_t, \quad 	F:=\langle\nabla\log\rho_t,\nabla u\rangle, \quad 	X:=A_tu.
$$
If $E=0$, then $u$, and hence, $w$ is constant, and the
inequality \eqref{e1.10} is immediate. Thus, we may assume
$E>0$.  By \eqref{h2.1}, we have 
$$
\int_M\left[\left|\nabla^2 u\right|^2+\Ric_V(\nabla u, \nabla u)\right] d\mu_t=\int_M X^2 d \mu_t-\int_M S_t(\nabla u, \nabla u) d \mu_t.
$$
 Therefore, for any $\beta \in\left(\beta_s, \frac{n}{n-1}\right)$, 
we may write
	\begin{align}
		\mathcal J=
		(1-\beta)\int_M X^2\,d\mu_t
		+\beta\int_M|\nabla^2u|^2\,d\mu_t+
		\int_M
		\mathcal{R}_{\beta, t}(\nabla u,\nabla u)\,d\mu_t
		+\tau \int_M \frac{|\nabla u|^2 A_t u}{u} d \mu_t,
	\end{align}
	where $\mathcal{R}_{\beta, t}=\beta \Ric_V+(\beta-1) S_t$. Since $\mathfrak{r}_\beta(t):=\inf _{\substack{x \in M \\|\xi|_g=1}}\left(\mathcal{R}_{\beta, t}\right)_x(\xi, \xi),$
 we imply that
	\begin{equation}\label{e3.28}
		\begin{aligned}
			\mathcal J
			&\geq
			\mathfrak r_\beta(t)E
			+(1-\beta)\int_M X^2\,d\mu_t+
			\beta\int_M|\nabla^2u|^2\,d\mu_t
			+\tau \int_M \frac{|\nabla u|^2 A_t u}{u} d \mu_t.
		\end{aligned}
	\end{equation}
From \eqref{r3.7}, we deduce that
\begin{align}\label{e3.29}
	(1-\beta) X^2+\beta\left|\nabla^2 u\right|^2=\frac{d_\beta}{n} X^2-\frac{2 \beta}{n} X F+\beta\left|\nabla^2 u-\frac{X}{n} g\right|^2.
\end{align}
where 
$d_\beta:=n-(n-1)\beta.$
By \eqref{e2.14}, we get
\begin{align}
		\int_M\frac{|\nabla u|^2A_tu}{u}\,d\mu_t =
	\frac{n}{n+2}\mathcal A
	-\frac{2n}{n+2}
	\int_M\langle\widetilde H_tu,Q(u)\rangle\,d\mu_t
	+\frac{2n}{n+2}D.\nonumber
\end{align}
where $Q(u)$ is defined as in \eqref{w3.1} and
$D:=
\int_M
\frac{|\nabla u|^2 \langle\nabla\log\rho_t,\nabla u\rangle}{u}\,d\mu_t.$ 
	Substituting this identity and \eqref{e3.29}
	into \eqref{e3.28}, we obtain
	\begin{align}\label{e3.32}
		\mathcal J
		&\geq
		\mathfrak r_\beta(t)E
		+\int_M
		\left(
		\frac{d_\beta}{n}X^2-\frac{2\beta}{n}XF
		\right)d\mu_t
		+\beta\int_M
		\left|\nabla^2u-\frac{X}{n}g\right|^2d\mu_t
		\nonumber\\
		&\quad-
		\frac{2n\tau}{n+2}
		\int_M\langle\widetilde H_tu,Q(u)\rangle\,d\mu_t
		+\frac{n\tau}{n+2}\mathcal A
		+\frac{2n\tau}{n+2}D.
	\end{align} 
We now define  
	$$Z:=\nabla^2u-\frac{X}{n}g
	-\frac{n\tau}{(n+2)\beta}Q(u).$$
It follows from \eqref{p3.11} and \eqref{w3.9} that
	\begin{align*}
		\beta\int_M|Z|^2\,d\mu_t
		&=
		\beta\int_M
		\left|\nabla^2u-\frac{X}{n}g\right|^2d\mu_t-
		\frac{2n\tau}{n+2}
		\int_M\langle\widetilde H_tu,Q(u)\rangle\,d\mu_t+
		\frac{2(n-1)\tau}{n+2}D
		+\frac{n(n-1)\tau^2}{(n+2)^2\beta}\mathcal A.
	\end{align*}
Plugging this information in \eqref{e3.32}, we conclude that
	\begin{align}\label{e3.33}
		\mathcal J
		&\geq
		\mathfrak r_\beta(t)E
		+\int_M
		\left(
		\frac{d_\beta}{n}X^2-\frac{2\beta}{n}XF
		\right)d\mu_t
		+\beta\int_M|Z|^2\,d\mu_t+
		K_{s,\beta}\mathcal A
		+\frac{2\tau}{n+2}D,
	\end{align}
	where 
	$$K_{s,\beta}
	:=
	\frac{n\tau}{n+2}
	-\frac{n(n-1)\tau^2}{(n+2)^2\beta}
	=
	\frac{n\tau}{n+2}
	\left(1-\frac{\beta_s}{\beta}\right)
	>0.$$
Using the same arguments as in proof of Theorem \ref{main1}, we conclude that
$$\begin{aligned}
	\mathcal J
&\geq
\left[
\mathfrak r_\beta(t)
+\frac{d_\beta}{n}x^2
-\frac{2\beta}{n}xy
+\left(\frac{\beta}{n}-\chi_{\beta_s,\beta}\right)y^2
\right]E
\geq \left[
\mathfrak r_\beta(t)
+
\widehat\Psi_{\beta,\beta_s}
\bigl(\lambda_1(t),H_t\bigr)
\right]E=
\widehat{\mathcal C}_t(\beta_s,\beta)E,
\end{aligned}$$
	where $
x:=\frac{\|X\|_{L^2\left(M, d \mu_t\right)}}{\sqrt{E}}, y:=\frac{\|F\|_{L^2\left(M, d \mu_t\right)}}{\sqrt{E}}.
$
	Together with \eqref{e3.25} and \eqref{w3.19},
	this proves \eqref{e1.10}.
\end{proof}
\begin{remark}
	The restriction on $s$ ensures that
	$0<\beta_s<\frac{n}{n-1},$
	so that the parameter interval used in the proof is
	nonempty. If
	$-\frac{2(2n^2+1)}{4n-1}<s<-2,$
	then $\beta_s>\frac{n}{n-1}$, and this argument
	provides no admissible choice of $\beta$.
	This obstruction is consistent with the counterexample
	of Cai-Wang \cite[Remark 4.1]{Cai}.
For each such $s$, they construct a smooth positive function $w$ on the unit round sphere $\mathbb S^n$ satisfying
$$
\int_{\mathbb{S}^n} w^s\left[\Gamma_2(w)-\frac{|\nabla w|^2}{w} \Delta w\right] d v<n \int_{\mathbb{S}^n} w^s|\nabla w|^2 dv.
$$
	Thus, the corresponding unweighted inequality cannot
	be extended to this range while retaining the
	coefficient $n$. This counterexample does not,
	by itself, exclude estimates with a smaller constant. At the boundary value
	$s=-\frac{2(2n^2+1)}{4n-1}$, the parameter interval
	also collapses, although the unweighted inequality
	with coefficient $n$ remains valid on $\mathbb S^n$
	by \cite[Theorem 3]{Cai}.
\end{remark}
\section{Entropy estimates and functional inequalities}
\label{s4}

\quad In this section, we derive Shannon and Tsallis entropy estimates and apply them to interpolation inequalities for evolving and stationary measures.
\subsection{Shannon and Tsallis entropy estimates}

\quad A logarithmic rescaling of the inequality in
Theorem \ref{main1} yields the following $\Gamma_2$
estimate associated with Shannon entropy.
\begin{corollary}\label{c4.1}
	Let $(M^n,g)$ be a closed, connected Riemannian manifold
	of dimension $n\geq2$, equipped with a smooth vector field
	$V\in\mathcal C^\infty(TM)$ and the family of probability
	measures $(d\mu_t)_{t\in[0,T]}$ defined by \eqref{k1.3}.
	With $a_n:=\frac{n-1}{n+2}$, define 
$\widehat C_{\mathrm{Sh},V}^{*}(t)
:=
\sup_{\alpha\in\left(a_n,\frac{n}{n-1}\right)}
\widehat{\mathcal C}_t(a_n,\alpha).$
	Then, every positive function $h\in\mathcal C^\infty(M)$ satisfies
	\begin{align}\label{e3.14}
		\int_M h\left[
		|\nabla^2\log h|^2
		+\Ric_V(\nabla\log h,\nabla\log h)
		\right]\,d\mu_t
		\geq
		\widehat C_{\mathrm{Sh},V}^{*}(t)
		\int_M h|\nabla\log h|^2\,d\mu_t
	\end{align}
	for all $t\in[0,T]$.
\end{corollary}
\begin{proof}[Proof of Corollary \ref{c4.1}]
	Fix $t\in[0,T]$ and
	$\alpha\in\left(a_n,\frac{n}{n-1}\right)$. Observe that 
$
\lim _{s \rightarrow-\infty} \alpha_s=a_n.
$
This shows that $\alpha>\alpha_s$ for all sufficiently negative $s$. 
Let $h\in \mathcal{C}^\infty(M)$ be positive and set 
$\varphi:=-\log h.$
For $s<0$ sufficiently large in absolute value, define 
$w_s:=1-\frac{\varphi}{s}.$
Since $M$ is compact, $w_s>0$ for all sufficiently negative $s$.
Moreover,
$\nabla w_s=-\frac{1}{s} \nabla \varphi, \nabla^2 w_s=-\frac{1}{s} \nabla^2 \varphi.$
	Applying Theorem \ref{main1} to $w_s$ and multiplying
	by $s^2$, we obtain
$$
\int_M w_s^s\left[\left|\nabla^2 \varphi\right|^2+\Ric_V(\nabla \varphi, \nabla \varphi)\right] d \mu_t \geq \widehat{\mathcal{C}}_t\left(\alpha_s, \alpha\right) \int_M w_s^s|\nabla \varphi|^2 d \mu_t.
$$
As $s\longrightarrow-\infty$, we see that
	$
	w_s^s=\left(1-\frac{\varphi}{s}\right)^s \longrightarrow e^{-\varphi}=h
	$
	uniformly on $M$.  
	In the definition of $\widehat\Psi_{\alpha,a}$,
	the dependence on $a$ occurs only through the term
	$-\chi_{a,\alpha}y^2$. Since $0\leq y\leq H_t$
	in the corresponding minimization problem, we have
	$\left|\widehat{\mathcal{C}}_t\left(\alpha_s, \alpha\right)-\widehat{\mathcal{C}}_t\left(a_n, \alpha\right)\right| \leq H_t^2\left|\chi_{\alpha_s, \alpha}-\chi_{a_n, \alpha}\right| .
	$
	Since $\alpha>a_n$, the right-hand side tends to
	zero as $s\longrightarrow-\infty$, which shows that
	$\widehat{\mathcal C}_t(\alpha_s,\alpha)
	\longrightarrow
	\widehat{\mathcal C}_t(a_n,\alpha).$  Therefore, we get
	\begin{align}
		\int_M h
		\left[
		|\nabla^2\log h|^2
		+\Ric_V(\nabla\log h,\nabla\log h)
		\right]\,d\mu_t
		\geq
		\widehat{\mathcal C}_t(a_n,\alpha)
		\int_M h|\nabla\log h|^2\,d\mu_t.
	\end{align}
Since $\alpha$ is arbitrary, taking the supremum yields
\eqref{e3.14} for every positive smooth $h$.
\end{proof}
\begin{remark}
	In the unweighted case $V=0, d \mu_t=\frac{d v}{\operatorname{Vol}_g(M)}$, 
	the reference density is constant, and hence, we get
$
A_t=\Delta, H_t=0, S_t=0,$ and $\lambda_1(t)=\lambda_1,
$
	where $\lambda_1$ denotes the first positive eigenvalue
	of $-\Delta$. 
	Assume that $\Ric\geq(n-1)g$. Since $S_t=0$ and
	$\Ric_V=\Ric$, the definition of $\mathfrak r_\alpha(t)$
	gives
	$\mathfrak r_\alpha(t)
	\geq
	\alpha(n-1).$
	Moreover, the constraint $0\leq y\leq H_t=0$
	forces $y=0$ in the joint minimization. Therefore,
	for every $\alpha\in\left(a_n,\frac{n}{n-1}\right)$, we have 
$
\widehat{\Psi}_{\alpha, a_n}\left(\lambda_1, 0\right)=\min _{x \geq \sqrt{\lambda_1}} \frac{d_\alpha}{n} x^2=\frac{d_\alpha}{n} \lambda_1,
$
	because $d_\alpha=n-(n-1)\alpha>0$. It follows that
	$$
	\widehat{\mathcal{C}}_t\left(a_n, \alpha\right) \geq\left(1-\frac{n-1}{n} \alpha\right) \lambda_1+\alpha(n-1) .
	$$
	The optimized coefficient is independent of $t$
	in this setting, and we denote it by
	$\widehat C_{\mathrm{Sh},0}^{*}$. Letting
	$\alpha\longrightarrow a_n=\frac{n-1}{n+2}$ yields 
	$\widehat C_{\mathrm{Sh},0}^{*}
	\geq
	\frac{4n-1}{n(n+2)}\lambda_1
	+
	\frac{(n-1)^2}{n+2}.$
	Substitution into \eqref{e3.14} recovers Ji's
	$\Gamma_2$ inequality \cite[Theorem 2.3]{Ji}.
	The normalization of the volume measure does not
	affect the constant, since it multiplies both sides
	of the inequality by the same factor.
	
	For a general reference measure, fix a time $t$
	at which $\rho_t$ is nonconstant on $M$.
	Since $M$ is connected, 
	$H_t>0.$
	In this case, 
	$\chi_{a_n,\alpha}
	=
	\frac{a_n\alpha}
	{n(n-1)(\alpha-a_n)}
	\longrightarrow+\infty$
as $\alpha\longrightarrow a_n.$
	To determine its effect on the joint minimum,
	evaluate the objective function at the admissible
	pair $(x,y)=(\sqrt{\lambda_1(t)},H_t)$. This gives
	\begin{align*}
		\widehat{\mathcal C}_t(a_n,\alpha)
		\leq{}&
		\frac{d_\alpha}{n}\lambda_1(t)
		-\frac{2\alpha}{n}H_t\sqrt{\lambda_1(t)}
		+\frac{\alpha}{n}H_t^2-\chi_{a_n,\alpha}H_t^2
		+\mathfrak r_\alpha(t).
	\end{align*}
	All terms on the right-hand side of the above inequality except
	$-\chi_{a_n,\alpha}H_t^2$ remain bounded as
	$\alpha\longrightarrow a_n$. Consequently, 
$\widehat{\mathcal C}_t(a_n,\alpha)
	\longrightarrow-\infty$
as $\alpha\longrightarrow a_n.$
	Thus, this endpoint limit does not provide a finite
	lower bound when $H_t>0$.
\end{remark}

We now prove Theorem \ref{main3} using the
forward-backward duality and the fixed-time spectral
$\Gamma_2$ estimates \eqref{e1.10}.

\begin{proof}[Proof of Theorem \ref{main3}]
		Set 
	$W_t:=V-\nabla\log\rho_t.$
	The equation \eqref{q1.2} becomes
	$\partial_t\rho_t=-\operatorname{div}(\rho_tW_t).$
	Since $M$ is closed, from the integration by parts, we deduce that
	\begin{align}\label{w4.6}
		\int_M F \Delta_V H d \mu_t=-\int_M\langle\nabla F, \nabla H\rangle d \mu_t+\int_M F\left\langle W_t, \nabla H\right\rangle d \mu_t
	\end{align}
	for all smooth functions $F$ and $H$. Moreover, we have
	\begin{align}\label{w4.7}
		\int_M F (\partial_t \rho_t )d v=-\int_M F \operatorname{div}\left(\rho_t W_t\right) d v=\int_M\left\langle W_t, \nabla F\right\rangle d \mu_t.
	\end{align}
	For every smooth time-dependent function $\psi$, one finds that
	\begin{align}\label{w4.15}
		\frac{d}{dt}\int_M\psi\,d\mu_t
		&=
		\int_M
		\left[\partial_t\psi+\langle W_t,\nabla\psi\rangle\right]
		\,d\mu_t=
		\int_M(\partial_t\psi+\Delta_V\psi)\,d\mu_t.
	\end{align}
By \eqref{w4.15} and $\partial_tu+\Delta_Vu=0$,
we have $\frac{d}{dt}\int_Mu\,d\mu_t=0$.
This implies that 
$$\int_Mu(\cdot,t)\,d\mu_t
=
\int_Mu_T\,d\mu_T
=1,
\quad t\in[0,T].$$

	We now consider the case $p\neq 1$. Using \eqref{w4.15}  and $\partial_tu+\Delta_Vu=0$, we obtain
	\begin{align}\label{t4.16}
		\frac{d}{dt}T_{p,V}(u,t)
		&=
		\frac{1}{1-p}
		\int_M
		\left[\partial_t(u^p)+\Delta_V(u^p)\right]d\mu_t
=
		-p\int_Mu^{p-2}|\nabla u|^2\,d\mu_t
		:=
		-I_p(t),
	\end{align}
	where 
	\vspace{-0.1cm}
	$$ I_p(t):=
	p\int_Mu^{p-2}|\nabla u|^2\,d\mu_t\geq 0.$$ We first consider 
	$p\in \left(\frac{2(n-1)^2}{2n^2+1}, 2\right)\setminus\{1\}.$ Set 
	$s:=\frac{2}{p-1},
	U:=u^{\frac{p-1}{2}}=u^{\frac1s}.$
	Then, $U>0$, $u=U^s$, and $s(p-1)=2$.
	If $\frac{2(n-1)^2}{2 n^2+1}<p<1$, then 
	$s<-\frac{2(2n^2+1)}{4n-1},$
	whereas $1<p<2$ gives $s>2$.
	Thus, $s$ belongs to the admissible range of
	Theorem \ref{main2}. Moreover,
	\begin{equation}\label{p4.17}
	\beta_s=\frac{n-1}{n+2} \frac{s-2}{s+2}=\frac{n-1}{n+2}\left(\frac{2}{p}-1\right)=\beta_p.
	\end{equation}
	Since $\nabla u=sU^{s-1}\nabla U$, we have
		$u^{p-2}|\nabla u|^2
	=s^2U^s|\nabla U|^2.$
	Consequently, writing 
	$\mathcal F(t):=
	\int_MU^s|\nabla U|^2\,d\mu_t,$
	we get
	\begin{equation}\label{eq:4.3}
		I_p(t)=ps^2\mathcal F(t),
		\quad
		\frac{d}{dt}T_{p,V}(u,t)
		=
		-ps^2\mathcal F(t).
	\end{equation}
	Set $G:=|\nabla U|^2$. By
	\eqref{w4.15} and the product rule,
	\begin{align}\label{e4.3}
		\mathcal F'(t)
		=
		\frac{d}{dt}\int_MuG\,d\mu_t
=
		\int_Mu
		\left[
		\Delta_VG
		+
		2\langle\nabla U,\nabla\partial_tU\rangle
		\right]d\mu_t
		+
		2\int_M\langle\nabla u,\nabla G\rangle\,d\mu_t,
	\end{align}
	where we used the fact that $\partial_tu+\Delta_Vu=0$. 
The equation 
$\partial_t U
=
-\Delta_V U
-(s-1)U^{-1}|\nabla U|^2,$
together with the Bochner-Weitzenb\"ock identity
gives that
	\begin{align}\label{e4.4}
		\Delta_VG
		+
		2\langle\nabla U, (\nabla\partial_tU)\rangle
		&=
		2\Gamma_{2,V}(U)
		-
		2(s-1)U^{-1}\langle\nabla U,\nabla G\rangle
		+
		2(s-1)U^{-2}G^2.
	\end{align}
	Substituting this into \eqref{e4.3} and using
	$\nabla u=sU^{s-1}\nabla U$, we find that
	\begin{align}\label{e4.5}
		\mathcal F'(t)
		&=
		2\int_MU^s\Gamma_{2,V}(U)\,d\mu_t
		+
		2\int_M
		U^{s-1}\langle\nabla U,\nabla G\rangle\,d\mu_t+2(s-1)\int_MU^{s-2}G^2\,d\mu_t.
	\end{align}
	Note that
	\begin{align}\label{e4.6}
		\int_MU^{s-1}G A_tU\,d\mu_t
		=
		-\int_M
		U^{s-1}\langle\nabla G,\nabla U\rangle\,d\mu_t
		-
		(s-1)\int_MU^{s-2}G^2\,d\mu_t.
	\end{align}
	Using this identity in \eqref{e4.5}, we conclude that
	\begin{equation}\label{e4.7}
		\mathcal F'(t)
		=
		2\int_MU^s
		\left[
		\Gamma_{2,V}(U)
		-
		\frac{|\nabla U|^2}{U}A_tU
		\right]d\mu_t.
	\end{equation}
	This, \eqref{eq:4.3} and \eqref{e4.7} lead to
	\begin{equation}\label{e4.41}
		\frac{d^2}{dt^2}T_{p,V}(u,t)
		=
		-2ps^2
		\int_MU^s
		\left[
		\Gamma_{2,V}(U)
		-
		\frac{|\nabla U|^2}{U}A_tU
		\right]d\mu_t.
	\end{equation}
	For each fixed $t$ and
	$\beta\in\left(\beta_p,\frac{n}{n-1}\right)$,
	Theorem \ref{main2} and
	\eqref{p4.17} imply that
$$
\int_M U^s\left[\Gamma_{2, V}(U)-\frac{|\nabla U|^2}{U} A_t U\right] d \mu_t \geq \widehat{\mathcal{C}}_t\left(\beta_p, \beta\right) \mathcal{F}(t).
$$
	Since $\mathcal F(t)\geq0$, taking the supremum over
	$\beta$ and using \eqref{eq:4.3}  and  \eqref{e4.41}
	yields that
	\begin{align*}
		\frac{d^2}{dt^2}T_{p,V}(u,t)
		&\leq
		-2ps^2\widehat\Lambda_{p,V}^{*}(t)\mathcal F(t)=
		2\widehat\Lambda_{p,V}^{*}(t)
		\frac{d}{dt}T_{p,V}(u,t).
	\end{align*}
	This proves \eqref{e1.7} for
	$p\in \left(\frac{2(n-1)^2}{2n^2+1}, 2\right)\setminus\{1\}$. 
	
	For $p=2$, we have
	$$T_{2,V}(u,t)=1-\int_Mu^2\,d\mu_t,
	  \quad
	I_2(t):=2\int_M|\nabla u|^2\,d\mu_t.$$
Using \eqref{t4.16}, we deduce that 
	$$\frac{d}{dt}T_{2,V}(u,t)=-I_2(t).$$
	Since $\partial_tu+\Delta_Vu=0$, the Bochner-Weitzenb\"ock identity
shows that
$$
\left(\partial_t+\Delta_V\right)|\nabla u|^2=2 \Gamma_{2, V}(u)+2\left\langle\nabla u, \nabla\left(\partial_t u+\Delta_V u\right)\right\rangle=2 \Gamma_{2, V}(u) .
$$
	Applying \eqref{w4.15},
	we obtain
$$
I_2^{\prime}(t)=2 \int_M\left(\partial_t+\Delta_V\right)|\nabla u|^2 d \mu_t=4 \int_M \Gamma_{2, V}(u) d \mu_t
$$
This and \eqref{p1.13} entail that
$$
I_2^{\prime}(t) \geq 4 \ell_t \int_M|\nabla u|^2 d \mu_t=2 \ell_t I_2(t).
$$
	Consequently,
	\[
	\frac{d^2}{dt^2}T_{2,V}(u,t)
	-
	2\ell_t\frac{d}{dt}T_{2,V}(u,t)
	=
	-I_2'(t)+2\ell_t I_2(t)
	\leq0.
	\]
	Since $\widehat\Lambda_{2,V}^*(t)=\ell_t$,
	this proves \eqref{e1.7} at $p=2$.
	
	For $p=1$, set
	\[
	\varphi:=\log u,
	  \quad
	I_1(t):=\int_Mu|\nabla\varphi|^2\,d\mu_t.
	\]
	Since $(\partial_t+\Delta_V)u=0$, the chain rule gives
	$(\partial_t+\Delta_V)\varphi=-|\nabla\varphi|^2$.
	The product rule and the Bochner-Weitzenb\"ock identity then yield that
	\begin{align*}
		(\partial_t+\Delta_V)(u\log u)
		=u|\nabla\varphi|^2,\quad
		(\partial_t+\Delta_V)\bigl(u|\nabla\varphi|^2\bigr)
		&=2u\Gamma_{2,V}(\varphi).
	\end{align*}
	Applying \eqref{w4.15}, we get 
	$$\frac{d}{dt}\mathcal S_V(u,t)=-I_1(t),
	  \quad
	I_1'(t)
	=
	2\int_Mu\Gamma_{2,V}(\log u)\,d\mu_t.$$
	Applying Corollary \ref{c4.1} to $u(\cdot,t)$, we conclude that
	$$I_1'(t)\geq
	2\widehat C_{\mathrm{Sh},V}^*(t)I_1(t).$$
	Consequently,
	\[
	\frac{d^2}{dt^2}\mathcal S_V(u,t)
	-
	2\widehat C_{\mathrm{Sh},V}^*(t)
	\frac{d}{dt}\mathcal S_V(u,t)
	=
	-I_1'(t)
	+2\widehat C_{\mathrm{Sh},V}^*(t)I_1(t)
	\leq0.
	\]
	Since $T_{1,V}=\mathcal S_V$ and
	$\widehat\Lambda_{1,V}^*(t)
	=\widehat C_{\mathrm{Sh},V}^*(t)$,
	this proves \eqref{e1.7} at $p=1$.
\end{proof}
\begin{remark}
	The differential inequality \eqref{e1.7} also holds
	for terminal data of arbitrary positive mass.
	Indeed, if $m:=\int_Mu_T\,d\mu_T>0$, the
	forward-backward duality gives
	$\int_Mu(\cdot,t)\,d\mu_t=m$ for all $t$.
	Thus, $v:=\frac{u}{m}$ solves the same backward equation
	and has unit mass. 
	For $p\neq1$, the definition of the Tsallis entropy
implies that
	$$T_{p,V}(u,t)
	=
	m^pT_{p,V}(v,t)
	+\frac{m^p-1}{1-p},$$
	whereas 
	$\mathcal S_V(u,t)
	=
	m\mathcal S_V(v,t)-m\log m.$
	Since $m$ is constant in time, both relations yield 
	$$\frac{d^j}{dt^j}T_{p,V}(u,t)
	=
	m^p\frac{d^j}{dt^j}T_{p,V}(v,t),
	 \quad j=1,2,$$
	with $T_{1,V}=\mathcal S_V$.
	Multiplying the inequality for $v$ by $m^p>0$
	therefore proves \eqref{e1.7} for $u$.
\end{remark}
\subsection{Interpolation inequalities}

\quad We now prove Corollary \ref{coro1.8} by combining the entropy
estimate in Theorem \ref{main3} with the initial
interpolation inequality \eqref{k1.16}.
\begin{proof}[Proof of Corollary \ref{coro1.8}]
	Fix
	$q\in\left[1,\frac{2n^2+1}{(n-1)^2}\right), 
	p:=\frac{2}{q},$
and
	$D(t):=\widehat{\mathcal D}_{q,V}^{*}(t).$ 
	Then, we see that
	$p\in\left(\frac{2(n-1)^2}{2n^2+1},2\right]$
and
	$D(t)=\widehat\Lambda_{p,V}^{*}(t).$
	Indeed, $\beta_p=b_q$ when $q>1$, while $q=1$
	corresponds to $p=2$ and $D(t)=\ell_t$.
	In particular, 
	$q=2,  p=1,$ and
	$D(t)=\widehat C_{\mathrm{Sh},V}^{*}(t).$ 
	
For $q=1$, \eqref{eq:1.11} makes $D=\ell_t$ upper
semicontinuous, while the spectral bounds proved in
Corollary~\ref{c2.2} ensure uniform boundedness on $[0,T]$. 
For $q>1$, the variational characterization of
$\lambda_1(t)$ and smoothness and positivity of $\rho_t$
give continuity of $\lambda_1(t)$, $H_t$, and
$\mathfrak r_\beta(t)$.
By Proposition~\ref{p1.3}, each
$t\mapsto\widehat{\mathcal C}_t(b_q,\beta)$ is continuous.
Thus, $D$ is lower semicontinuous, and any fixed
admissible $\beta_0$ gives 
$D(t)\geq
\min_{r\in[0,T]}
\widehat{\mathcal C}_r(b_q,\beta_0)>-\infty.$
Hence, $D$ is Borel measurable and uniformly bounded
below in both cases.
	
Fix a positive nonconstant
	$h_*\in\mathcal C^\infty(M)$.
	For $1<q<\frac{2 n^2+1}{(n-1)^2}, q \neq 2$, set $s=\frac{2 q}{2-q}$. For $q=1$, we use instead the variational characterization of $\ell_t$.
	Then, $s$ lies in the admissible range of
	Theorem \ref{main2}, and $\beta_s=b_q$.
	Applying Theorem \ref{main2} to $h_*$ and taking the
	supremum over $\beta$, we obtain 
\begin{align}
D(t)\leq
\frac{
	\displaystyle
	\int_Mh_*^s
	\left[
	\Gamma_{2,V}(h_*)
	-\frac{|\nabla h_*|^2}{h_*}A_th_*
	\right]\,d\mu_t
}{
	\displaystyle
	\int_Mh_*^s|\nabla h_*|^2\,d\mu_t
}.
\end{align}
For $q=2$, applying Corollary~\ref{c4.1} to the fixed
positive nonconstant function $h_*$ likewise provides
a continuous upper bound for $D$ on $[0,T]$. 
	Each quotient is continuous on $[0,T]$ and has a
	strictly positive denominator, since $\rho_t>0$
	and $\nabla h_*\not\equiv0$.
	Compactness therefore supplies a uniform upper
	bound for $D$.
	Thus, $D$ is bounded and Borel measurable in all
	cases, and \eqref{q1.19} defines
	$0<c_{q,T}<\infty$.
	
	Now fix a positive function $w\in\mathcal C^\infty(M)$ and set 
	$Z_T:=\int_Mw^q\,d\mu_T,
	u_T:=\frac{w^q}{Z_T}.$
	Let $u>0$ be the smooth solution of \eqref{eq:1.9} with terminal
	datum $u_T$.
Using \eqref{w4.15}, we imply that 
	$$\int_Mu(\cdot,t)\,d\mu_t
	=
	\int_Mu_T\,d\mu_T
	=1,
	 \quad 0\leq t\leq T.$$
	Define 
	$$\mathcal E_p(t):=
	\begin{cases}
		\displaystyle
		\frac{	\displaystyle\int_Mu^p\,d\mu_t-1}{p-1},
		& p\neq1,\\[7pt]
		\displaystyle
		\int_Mu\log u\,d\mu_t,
		& p=1,
	\end{cases}
	 \quad
	I_p(t):=
	p\int_Mu^{p-2}|\nabla u|^2\,d\mu_t.$$
	Thus, $\mathcal E_p=-T_{p,V}$, with the convention
	$T_{1,V}=\mathcal S_V$.
From 
	\eqref{e1.7} and \eqref{eq:4.3}, we find that 
	$\mathcal E_p'(t)=I_p(t)$
and
	$I_p'(t)\geq2D(t)I_p(t).$ 
	Since $D$ is bounded and Borel measurable, the function
	$t\mapsto e^{-2\int_0^tD(r)\,dr}I_p(t)$ is
	absolutely continuous, with
	\[
	\frac{d}{dt}
	\left[e^{-2\int_0^tD(r)\,dr}I_p(t)\right]
	=
	e^{-2\int_0^tD(r)\,dr}
	\left[I_p'(t)-2D(t)I_p(t)\right]
	\geq0
	\]
	for almost every $t\in(0,T)$.
	Consequently,
	\begin{equation}\label{r4.29}
		I_p(t)
		\leq
		\exp\left(-2\int_t^TD(r)\,dr\right)I_p(T),
		 \quad 0\leq t\leq T.
	\end{equation}
	Set $u_0:=u(\cdot,0)$ and $v_0:=u_0^{\frac1q}$.
	Then, $\int_Mv_0^q\,d\mu_0=1$.
	For $q\neq2$, the fact that $p=\frac{2}{q}$ gives 
	$\mathcal B_{q,\mu_0}(v_0)
	=
	\frac{\mathcal E_p(0)}{q}.$
	The same identity holds at $q=2$, since 
	$\mathcal B_{2,\mu_0}(v_0)
	=
	\frac12\operatorname{Ent}_{\mu_0}(u_0)
	=
	\frac{\mathcal E_1(0)}{2}.$
	Moreover, for every admissible $q$,
	\[
	\int_M|\nabla v_0|^2\,d\mu_0
	=
	\frac{1}{q^2}
	\int_Mu_0^{p-2}|\nabla u_0|^2\,d\mu_0
	=
	\frac{I_p(0)}{2q}.
	\]
	Applying \eqref{k1.16} to $v_0$, we therefore obtain 
	$\mathcal E_p(0)\leq\frac{I_p(0)}{2c_{q,0}}.$
	Integrating $\mathcal E_p'=I_p$ and using
	\eqref{r4.29}, we conclude that
	\begin{equation}\label{k4.32}
		\begin{aligned}
			\mathcal E_p(T)
			&=
			\mathcal E_p(0)+\int_0^TI_p(t)\,dt\\
			&\leq
			\left[
			\frac{1}{2c_{q,0}}
			\exp\left(-2\int_0^TD(r)\,dr\right)
			+
			\int_0^T
			\exp\left(-2\int_t^TD(r)\,dr\right)\,dt
			\right]I_p(T)=
			\frac{I_p(T)}{2c_{q,T}},
		\end{aligned}
	\end{equation}
	where the last equality follows from \eqref{q1.19}.
	
	At the terminal time,
	$u_T^{\frac1q}=Z_T^{-\frac1q}w$.
	The preceding normalization identities, together
	with the homogeneity
	$\mathcal B_{q,\mu}(cw)=c^2\mathcal B_{q,\mu}(w)$
	for $c>0$, deduce that
$$
\mathcal{E}_p(T)=\frac{q}{Z_T^{\frac{2}{q}}} \mathcal{B}_{q, \mu_T}(w), \quad I_p(T)=\frac{2 q}{Z_T^{\frac{2}{q}}} \int_M|\nabla w|^2 d \mu_T .
$$
	These formulas also hold at $q=2$.
	Substituting them into \eqref{k4.32} and cancelling
	the positive factor $\frac{q}{Z_T^{\frac{2}{q}}}$
	proves \eqref{e1.8}. At $q=2$, the logarithmic Sobolev inequality also holds
	for every real-valued smooth function $w$.
	To see this, apply it to $\sqrt{w^2+\varepsilon}$
	and let $\varepsilon\longrightarrow 0$.
\end{proof}

We now prove Corollary \ref{c1.10}
by specializing Theorem \ref{main3} to the Witten
Laplacian and its invariant probability measure $d\nu_f$.
\begin{proof}[Proof of Corollary \ref{c1.10}]
	We apply Theorem \ref{main3} with 
$V:=-\nabla f, 
\rho:=\frac{e^{-f}}{\int_M e^{-f} d v},$
and
$d\mu_t\equiv d\nu_f.$ 
	Since $\nabla\rho=\rho V$, the density $\rho$ is
	stationary for the Fokker-Planck equation \eqref{q1.2}. Moreover,
$$
A_t=\Delta_f, \quad S_t=0, \quad \Ric_V=\Ric_f, \quad \lambda_1(t)=\lambda_{1, f}, \quad H_t=H:=\|\nabla f\|_{L^{\infty}(M)} \leq L .
$$
	Fix $q\in\left(1,\frac{2n^2+1}{(n-1)^2}\right)$ and set 
	$p:=\frac{2}{q}, 
	D:=\widehat{\mathcal D}_{q,f}^*.$
	Then, $\beta_p=b_q$. Since $\mathfrak r_\beta(t)\geq\beta\kappa$
	and $\widehat\Psi_{\beta,b_q}(\lambda,\cdot)$ is
	nonincreasing, every admissible $\beta$ satisfies
	$$\widehat{\mathcal C}_t(b_q,\beta)
	\geq
	\widehat\Psi_{\beta,b_q}(\lambda_{1,f},L)
	+\beta\kappa.$$
	Thus, $D\leq\widehat\Lambda_{p,V}^*(t)<\infty$,
	while a fixed admissible $\beta$ gives $D>-\infty$. 
	For $q\neq2$, the assertion is immediate if $D\leq0$,
	since $\mathcal B_{q,\nu_f}(w)\geq0$.
	We therefore assume $D>0$, which also covers the
	logarithmic Sobolev assertion at $q=2$. 
	Fix a positive $w\in\mathcal C^\infty(M)$ and set 
	$Z:=\int_Mw^q\,d\nu_f, 
	h_0:=\frac{w^q}{Z}.$
	Let $h$ solve $\partial_\tau h=\Delta_fh$ with
	$h(\cdot,0)=h_0$. Mass conservation and the maximum
	principle give
	\[
	\int_Mh(\cdot,\tau)\,d\nu_f=1,
	\quad
	0<\min_Mh_0\leq h(x,\tau)\leq\max_Mh_0.
	\]
	Define
	\[
	\mathcal E_p(\tau):=
	\begin{cases}
		\displaystyle
		\frac{	\displaystyle\int_Mh^p\,d\nu_f-1}{p-1},
		&p\neq1,\\[6pt]
		\displaystyle
		\int_Mh\log h\,d\nu_f,
		&p=1,
	\end{cases}
	\quad
	I_p(\tau):=
	p\int_Mh^{p-2}|\nabla h|^2\,d\nu_f.
	\]
	Applying the first-variation identity and
	Theorem \ref{main3} to $u_R(x,t):=h(x,R-t)$,
	with arbitrary $R>0$, and reversing time yields 
	$\mathcal E_p'(\tau)=-I_p(\tau)$
and
	$I_p'(\tau)\leq-2D I_p(\tau).$
	Consequently,
	\begin{equation}\label{q4.33}
		I_p(\tau)\leq e^{-2D\tau}I_p(0),
		\quad \tau\geq0.
	\end{equation}
	The spectral gap gives 
	$$\frac{d}{d\tau}\|h-1\|_{L^2(\nu_f)}^2
	=
	-2\int_M|\nabla h|^2\,d\nu_f
	\leq
	-2\lambda_{1,f}\|h-1\|_{L^2(\nu_f)}^2.$$
	Hence, $h(\cdot,\tau)\longrightarrow1$ in $L^2(\nu_f)$.
	The uniform positive bounds on $h$ imply
	$\mathcal E_p(\tau)\longrightarrow0$, since the entropy integrands
	are Lipschitz on $[\min_Mh_0,\max_Mh_0]$.
	Integrating $\mathcal E_p'=-I_p$ and using
	\eqref{q4.33}, we obtain
	\begin{equation}\label{q4.34}
		\mathcal E_p(0)
		=
		\int_0^\infty I_p(\tau)\,d\tau
		\leq
		\frac{I_p(0)}{2D}.
	\end{equation}
	The normalization identities from the proof of
	Corollary \ref{coro1.8}, also valid at $q=2$, give 
	$$\mathcal E_p(0)
	=
	\frac{q}{Z^{\frac{2}{q}}}\mathcal B_{q,\nu_f}(w),
	\quad
	I_p(0)
	=
	\frac{2q}{Z^{\frac{2}{q}}}
	\int_M|\nabla w|^2\,d\nu_f.$$
	Substitution into \eqref{q4.34} yields
	\[
	\int_M|\nabla w|^2\,d\nu_f
	\geq D\,\mathcal B_{q,\nu_f}(w).
	\]
	This proves \eqref{p1.22} for $q\neq2$ and
	\eqref{q1.23} for $q=2$.
	
	Finally, \eqref{e1.6} follows from the variational
	characterization of $\lambda_{1,f}$ applied to
	$w-\int_Mw\,d\nu_f$.
	Sharpness follows by taking $w=K+\psi$, where
	$\psi\not\equiv0$ satisfies
	$-\Delta_f\psi=\lambda_{1,f}\psi$ and
	$K>\|\psi\|_{L^\infty(M)}$.
\end{proof}

\begin{remark}
	In the stationary Witten setting,
	$D_q:=\widehat{\mathcal D}_{q,f}^*$
	is a time-independent lower bound for the Fisher
	coefficient. If $D_q>0$, then \eqref{q1.19} gives 
	$c_{q,T}\geq
	\left[
	\frac{1-e^{-2D_qT}}{D_q}
	+\frac{e^{-2D_qT}}{c_{q,0}}
	\right]^{-1}.$
	The right-hand side converges to $D_q$ as
	$T\longrightarrow\infty$, yielding the stationary interpolation
	inequality. At $q=2$, the Shannon case gives
	\eqref{q1.23} directly. 
	If $f$ is constant and $\Ric\geq(n-1)g$, choose
	$L=0$ and $\kappa=n-1$.
	Then, $\lambda_{1,f}=\lambda_1\geq n$ by
	Lichnerowicz's estimate, so the expression below
	is nonincreasing in $\beta$.
	Taking $\beta\longrightarrow b_q$, we obtain
	\begin{align*}
		D_q
		&=
		\sup_{\beta\in \left(b_q, \frac{n}{n-1}\right)}
		\left[
		\lambda_1+\beta(n-1)
		\left(1-\frac{\lambda_1}{n}\right)
		\right]=
		\left(
		1-\frac{(n-1)^2(q-1)}{n(n+2)}
		\right)\lambda_1
		+\frac{(n-1)^2(q-1)}{n+2}.
	\end{align*}
For $1<q<\frac{2 n^2+1}{(n-1)^2}$ with $q \neq 2$, this recovers the coefficient in
	\cite[Corollary 1.1]{Cai}.
\end{remark}
\section*{Acknowledgments}
 \quad The second author is funded by the Shanghai Institute for Mathematics and Interdisciplinary Sciences (SIMIS) under grant
number SIMIS-ID-2025-AD. 
\section*{Appendix}
\quad The existence, uniqueness, and smoothness of solutions to linear parabolic equations on closed Riemannian manifolds are well established in the literature; see, for instance \cite{Am, CM}. Likewise, the positivity preserving and mass conserving properties of Fokker-Planck equations on compact manifolds are classical \cite{Yos}. For completeness, we provide a proof of these properties in the present setting, together with explicit upper and lower bounds on finite time intervals for the smooth positive probability densities used in our entropy arguments.

\begin{proposition}\label{p4.7}
	Let $(M,g)$ be a closed Riemannian manifold and
	$V\in\mathcal C^\infty(TM)$. Suppose that
	$$
	\rho_0 \in \mathcal{C}^{\infty}(M), \quad \rho_0>0, \quad \int_M \rho_0 d v=1.
	$$
	For every $T>0$, the Cauchy problem
	\begin{equation}\label{e4.35}
		\begin{cases}
			\quad\, \partial_t\rho
			=
			\Delta\rho-\operatorname{div}(\rho V),
			&\text{on }M\times(0,T),\\
			\rho(\cdot,0)=\rho_0,
			&\text{on }M,
		\end{cases}
	\end{equation}
	admits a unique solution
	$\rho\in\mathcal C^\infty(M\times[0,T])$.
Moreover, we have
	\begin{equation}\label{e4.45}
	m_0 e^{-D t} \leq \rho(x, t) \leq M_0 e^{D t}, \quad(x, t) \in M \times[0, T],
	\end{equation}
	and
	\begin{equation}\label{e4.46}
		\int_M\rho(x,t)\,dv=1,
		\quad 0\leq t\leq T,
	\end{equation}
	where $$
	m_0:=\min _M \rho_0, \quad M_0:=\max _M \rho_0, \quad D:=\|\operatorname{div} V\|_{L^{\infty}(M)}.
	$$
\end{proposition}

\begin{proof}[Proof of Proposition \ref{p4.7}]
	Fix $T>0$ and write $d:=\operatorname{div}V$.
	The equation \eqref{e4.35} takes the form 
	$$\partial_t\rho
	=
	\Delta\rho-\langle V,\nabla\rho\rangle-d\rho.$$
The principal symbol of the second-order spatial part
is $|\xi|_g^2$, so the equation is uniformly parabolic.
Its coefficients are smooth and independent of time. Since $M$ is compact,
	the coefficients and all their covariant derivatives are
	bounded. The linear parabolic result
	\cite[Proposition 2.3 and Remark 2.4]{CM},
	applied with order parameter $p=1$, therefore gives a
	solution 
	$\rho\in\mathcal C^\infty(M\times[0,T]).$
	This includes smoothness at $t=0$ and $t=T$, and the
	result applies to every prescribed finite $T>0$.
	
 If $\rho_1$ and $\rho_2$ are smooth solutions
	with the same initial data, set 
	$$z:=\rho_1-\rho_2,
	\quad
	E(t):=\int_Mz^2\,dv.$$
	Then,
	$\partial_tz=\Delta z-\operatorname{div}(zV)$
	and $z(\cdot,0)=0$. The integration by parts gives
	\begin{align*}
		E'(t)
		=
		-2\int_M|\nabla z|^2\,dv
		+
		2\int_Mz\langle V,\nabla z\rangle\,dv=
		-2\int_M|\nabla z|^2\,dv
		-
		\int_Md\,z^2\,dv\leq DE(t).
	\end{align*}
	Since $E(0)=0$, the Gronwall's inequality implies
	$E(t)=0$ for all $t\in[0,T]$, and hence
	$\rho_1=\rho_2$.
	
	We now introduce 
	$\mathcal P h
	:=
	\partial_th-\Delta h+\langle V,\nabla h\rangle+dh,$
	so that $\mathcal P\rho=0$.
	We first show that if
	$h\in\mathcal C^\infty(M\times[0,T])$ satisfies
	$\mathcal P h\geq0$ on $M\times[0,T]$ and
	$h(\cdot,0)\geq 0$ on $M$, then
	\begin{equation}\label{h4.38}
		h\geq0
		\,\,\text{on }M\times[0,T].
	\end{equation}
	Set $\widetilde h:=e^{-(D+1)t}h$. Observe that
	\begin{equation}\label{q4.39}
		\begin{aligned}
			&\partial_t\widetilde h-\Delta\widetilde h
			+\langle V,\nabla\widetilde h\rangle
			+(d+D+1)\widetilde h
			=e^{-(D+1)t}\mathcal P h
			\geq0.
		\end{aligned}
	\end{equation}
	Here $d=\operatorname{div}V$ and
	$D=\|d\|_{L^\infty(M)}$, so $d+D+1\geq1$. 
	Suppose that $\widetilde h$ is negative somewhere.
	By continuity and compactness, there exists
	$(x_*,t_*)\in M\times[0,T]$ such that 
	$\widetilde h(x_*,t_*)
	=
	\min_{(x,t)\in M\times[0,T]}
	\widetilde h(x,t)
	<0.$
	Since $\widetilde h(\cdot,0)=h(\cdot,0)\geq0$,
	we have $0<t_*\leq T$.
	At this minimum point,
	$\nabla\widetilde h(x_*,t_*)=0,
	\Delta\widetilde h(x_*,t_*)\geq0,$ and $\partial_t\widetilde h(x_*,t_*)\leq 0.$
	Using these inequalities in \eqref{q4.39} gives 
	$0
	\leq
	\bigl(d(x_*)+D+1\bigr)\widetilde h(x_*,t_*)
	<0,$
	a contradiction.
	Thus, $\widetilde h\geq0$, and hence
	$h=e^{(D+1)t}\widetilde h\geq0$.
	This proves \eqref{h4.38}.

	We next consider the following functions
	$\rho_-(t):=m_0e^{-Dt}$
and
	$\rho_+(t):=M_0e^{Dt}.$
	A direct calculation gives
	\[
	\mathcal P\rho_-
	=
	(d-D)\rho_-\leq0,
	\quad
	\mathcal P\rho_+
	=
	(d+D)\rho_+\geq0.
	\]
	Consequently,
	$\mathcal P(\rho-\rho_-)\geq0,
	\mathcal P(\rho_+-\rho)\geq0,$
	and both differences are nonnegative at $t=0$.
	Applying \eqref{h4.38} yields 
	$\rho_-(t)\leq\rho(x,t)\leq\rho_+(t),$
	which proves \eqref{e4.45}.
	In particular, 
	$\rho(x,t)\geq m_0e^{-DT}>0$
on $M\times[0,T].$
	Finally, we see that
	\begin{align*}
		\frac{d}{dt}\int_M\rho(x,t)\,dv
	=
		\int_M\left[
		\Delta\rho-\operatorname{div}(\rho V)
		\right]\,dv=
		\int_M\operatorname{div}(\nabla\rho-\rho V)\,dv=0,
	\end{align*}
	because $M$ is closed. Together with the initial
	normalization, this proves
	\eqref{e4.46}.
\end{proof}

\begin{remark}
	Let $(M,g)$ be closed, $V\in\mathcal C^\infty(TM)$,
	and $T>0$. For every strictly positive
	$u_T\in\mathcal C^\infty(M)$, problem \eqref{eq:1.9}
	admits a unique solution
	$u\in\mathcal C^\infty(M\times[0,T])$ satisfying 
	$0<\min_Mu_T\leq u(x,t)\leq\max_Mu_T,$
	$(x,t)\in M\times[0,T].$
	Indeed, the time reversal $v(x,\tau)=u(x,T-\tau)$
	transforms \eqref{eq:1.9} into
	$\partial_\tau v=\Delta_Vv$, with $v(\cdot,0)=u_T$.
	Existence, uniqueness, and smoothness follow from
	linear parabolic theory
	\cite[Proposition 2.3 and Remark 2.4]{CM},
	and the bounds follow from the maximum principle.
\end{remark}

	\bigskip
\address{{\it Ha Tuan Dung}\\
Shanghai Institute for Mathematics  \\
and Interdisciplinary Sciences (SIMIS),\\
 Shanghai, China \\
and Research Institute of Intelligent\\ Complex
Systems, Fudan University,\\ Shanghai,
China
}
{hatuandung@simis.cn or\\ hatuandung@hpu2.edu.vn}
\address{{\it Yi Li}\\
	Shanghai Institute for Mathematics  \\
and Interdisciplinary Sciences (SIMIS),\\ Shanghai, China \\
and Center for Mathematics and\\
Interdisciplinary Sciences,  \\
Fudan
University, Shanghai, China
}
{ yilicms@simis.cn}
	%

%

\end{document}